\documentclass{amsart}

\usepackage[english]{babel}
\useshorthands{"}
\addto\extrasenglish{\languageshorthands{ngerman}}
\usepackage{amsmath}
\allowdisplaybreaks
\usepackage{amsfonts}
\usepackage[toc,page]{appendix}
\usepackage[latin1]{inputenc}
\usepackage[T1]{fontenc}
\usepackage{amssymb}
\usepackage{gensymb,MnSymbol}
\usepackage{algorithm}
\usepackage{algpseudocode} 
\usepackage{color}
\usepackage{capt-of}
\usepackage[skip=2pt,font=footnotesize]{caption}
\usepackage{xfrac}
\usepackage{hyperref}

\usepackage{pgfplots}
\pgfplotsset{compat=1.12}
\usepackage{paralist}
\usepackage{subfigure}

\usepackage{cite}
\usepackage{amsthm}
\usepackage{latexsym}
\usepackage{graphicx}

\usepackage{bbm}
\usepackage{url}
\usepackage{longtable}

\newcommand\oneD{{1\text{D}}}
\newcommand\twoD{{2\text{D}}}

\makeatletter
\DeclareRobustCommand*{\bbl@ap}[1]{\textormath{\textsuperscript{#1}}{^{\mathrm{#1}}}}%
\DeclareRobustCommand*{\bbl@ped}[1]{\textormath{$_{\mbox{\fontsize\sf@size\z@ \selectfont#1}}$}{_\mathrm{#1}}}%
\let\ap\bbl@ap
\let\ped\bbl@ped
\makeatother

\newtheorem{theorem}{Theorem}[section]
\newtheorem{definition}[theorem]{Definition}
\newtheorem{proposition}[theorem]{Proposition}
\newtheorem{lemma}[theorem]{Lemma}
\newtheorem{assumption}[theorem]{Assumption}
\newtheorem{remark}[theorem]{Remark}
\newtheorem{corollary}[theorem]{Corollary}
\newtheorem{example}[theorem]{Example}

\newcommand{\supp}{\operatorname{supp}}

\newcommand{\dom}{\operatorname{dom}}
\newcommand{\ran}{\operatorname{ran}}
 
\newcommand{\spanlin}{\operatorname{span}} 
\newcommand{\clos}{\operatorname{clos}}  
\newcommand{\la}{\langle} 
\newcommand{\ra}{\rangle}
\newcommand{\R}{\mathbb{R}} 
\newcommand{\N}{\mathbb{N}}

\newcommand{\cL}{{\mathcal{L}}}

\newcommand{\cX}{{\mathcal{X}}}
\newcommand{\cXd}{{\mathcal{X}^{\delta}}}
\newcommand{\bcXd}{{\bar{\mathcal{X}}^{\delta}}}
\newcommand{\cY}{{\mathcal{Y}}}
\newcommand{\cYd}{{\mathcal{Y}^{\delta}}}

\newcommand{\cP}{{\mathcal{P}}}

\numberwithin{equation}{section}

\title[(Parametrized) First Order Transport Equations]{(Parametrized) First Order Transport Equations: Realization of Optimally Stable Petrov-Galerkin Methods}
\author{Julia Brunken}
\address{University of M{\"u}nster, Applied Mathematics, 
	Einsteinstr.\ 62, 48149 M{\"u}nster (Germany), {julia.brunken@uni-muenster.de}}
\author{Kathrin Smetana}
\address{University of Twente, Faculty of Electrical Engineering, Mathematics \& Computer Science,
	Zilverling,  P.O. Box 217, 7500 AE Enschede (The Netherlands), {k.smetana@utwente.nl}}
\author{Karsten Urban}
\address{Ulm University, Institute for Numerical Mathematics, Helmholtzstr.\ 20, 
		89081 Ulm (Germany), {karsten.urban@uni-ulm.de}}
\date{\today}
\thanks{The work of Julia Brunken is supported by the German Federal Ministry of Education and Research under Grant no.\ BMBF 05M2016 - GlioMaTh.}


\begin{document}

\begin{abstract}
We consider ultraweak variational formulations for (parametrized) linear first order transport equations in time and/or space. Computationally feasible pairs of optimally stable trial and test spaces are presented, starting with a suitable test space and defining an optimal trial space by the application of the adjoint operator. As a result, the inf-sup constant is one in the continuous as well as in the discrete case and the computational realization is therefore easy. In particular, regarding the latter, we avoid a stabilization loop within the greedy algorithm when constructing reduced models within the framework of reduced basis methods. 
Several numerical experiments demonstrate the good performance of the new method.
\end{abstract}

\maketitle

\section{Introduction}\label{Sec:1}
Transport phenomena are omnipresent in several areas of science and technology such as cell movement \cite{Hillen2006} for instance in brain tumors \cite{EHKS2015}. 
Even though there is a huge literature on the corresponding partial differential equations (PDEs), there is still significant need for research, in particular concerning efficient and robust numerical solvers. Even less results are known for parametrized PDEs when one either wishes to compute the solution for many different parameters (many-query) or in real-time. Of course, in the simplest case of first order linear transport problems with constant coefficients, there are even closed formulas for the solution using the method of characteristics. This, however, already changes when allowing for variable advection and/or reaction coefficients,  which we encounter for instance in mesoscopic formulations of Glioma spreading models \cite{EHKS2015}. 
Such a model contains patient-specific data as parameters. To use this problem for individual cancer treatment planning it has to be solved numerically reasonably fast for given parameter values. 
This is the background why in this paper, as a first step, we are concerned with the simplified model problem of (parametrized) time-dependent linear first-order transport:
\begin{equation}\label{eq:TE}
\dot{u}_{\mu}(t,x) + \vec{b}_{\mu}(t,x)\cdot \nabla u_{\mu}(t,x) + c_{\mu}(t,x)\, u_{\mu}(t,x) = f_{\mu}(t,x),
\end{equation}
for all parameters $\mu$ in a compact set $\mathcal{P} \subset \mathbb{R}^{p}$, for all times $t\in (0,T)$ ($T>0$ being some final time) and all $x\in D\subset \R^d$ accompanied with appropriate initial and boundary conditions. 

It is well-known that the above point-wise formulation of \eqref{eq:TE} does not make sense for several realistic cases of coefficients, initial and/or boundary conditions, the geometry of $D$, etc. In fact, often it is known that continuous solutions of \eqref{eq:TE} do not exist; appropriate variational formulations are a possible way-out.

Recalling d'Alembert's solution formula for the linear transport equation \cite{MR2407757}, it is well-known that the solution ``inherits'' the (lack of) regularity of the initial condition, which means for instance that the solution stays in $L_2(D)$ if the initial condition is only in $L_2(D)$ (and not more). This motivates us to consider an ultraweak space-time formulation with $\cX=L_2(I; L_2(D))\cong L_2(I\times D)$ as trial space. It then remains to determine a test space $\cY$ such that the arising variational problem is well-posed. Besides existence and uniqueness of the solution, the stability is of particular interest for numerical purposes. 
In the optimal case the stability constant is unity, which means that error and residual coincide and at the same time the approximation is the best one from the chosen trial space. This is highly relevant for error estimation in adaptive methods and reduced order models.

Such optimally stable ultraweak variational formulations for first order transport equations have been proposed for instance in \cite{DemGop10,DemGop11,DHSW2012,BTDeGh13}. The optimal relation between trial and test spaces is thus known. For numerical purposes, however, this relation is not easy to deal with. In fact, given a finite-dimensional approximation trial space $\cX^\delta\subset\cX$, where $\delta$ is some discretization parameter such as the mesh size, the numerical construction of the optimal test space $\cY^\delta$ amounts to solving the PDE $\dim(\cX^\delta)$-times, which is in general infeasible. Therefore, in \cite{DemGop10,DemGop11,BTDeGh13,Zetal11,BrDaSt17} a discontinuous Petrov-Galerkin (DPG) approximation with a (possibly suboptimal) broken test space is suggested so that the approximation of the optimal test basis functions reduces to the solution of local problems. In \cite{DHSW2012}, a global approximate test space $\cY^{\delta} \subsetneq \mathcal{Z}^{\delta} \subset \cY$ is constructed by using an appropriate so-called test search space $\mathcal{Z}^{\delta}$ similar to \cite{DemGop11}. Finally, the authors of \cite{DFW16} employ a discontinuous Galerkin approximation in space and a conforming Petrov-Galerkin approximation in time resulting in a suboptimal inf-sup constant, in particular w.r.t.\ time \cite[Lemmata 1 and 3]{DFW16}.

In this article we propose to first choose an appropriate test space $\cY^\delta$ and subsequently compute the corresponding trial space $\cX^\delta$. Doing so, the optimal trial space $\cX^\delta$ arises from the \emph{application of the differential operator} on the basis functions of $\cY^\delta$ rather than by \emph{approximately solving (local) PDEs}. If the test space is chosen for instance as a standard finite element (FE) space, the application of the differential operator is straightforward and by far more efficient than computing approximate test functions. Moreover, the approach is (very) easy to implement. In contrast to the approaches mentioned above we obtain an optimally stable scheme, meaning inf-sup and continuity constants of unity also for variable coefficients. In particular, the inf-sup constant does not depend on $\delta$. We also prove convergence for our scheme as $\delta\to 0$ but do not derive convergence rates in $\delta$. Instead, we investigate the achieved rates numerically, obtaining convergence rates similar to the ansatz proposed in \cite{DHSW2012}. We also believe that our approach relatively easily might be generalizable to more complex problems. Finally, we note that choosing first the test space and constructing subsequently the associated optimal trial space was already suggested in \cite[Theorem 2.10]{BTDeGh13} for the DPG method but not further pursued in the remainder of the respective article. Moreover, the same approach is investigated in parallel for the wave equation in \cite{RB:Wave}.

Generalizing and applying our proposed approach to parametrized PDEs offers additional advantages. To realize the generalization we make use of the reduced basis (RB) method (see for instance \cite{QuMaNe16,HeRoSt16,Haasdonk:RB} and references therein), which is nowadays a well-known and accepted efficient numerical method for solving parametrized PDEs in a many-query and/or realtime context. For instance, we employ a greedy algorithm for the construction of reduced test and trial spaces. Here, by applying the now parameter-dependent operator to the reduced test space in order to construct the then also parameter-dependent reduced trial space we obtain an optimally stable reduced scheme. In contrast, the approaches proposed in \cite{DaPlWe14,ZaNo16} yield a suboptimal inf-sup constant. Moreover, we avoid an additional stabilization loop during the greedy algorithm as proposed in \cite{DaPlWe14} or the construction of a parameter-dependent preconditioner as suggested in \cite{ZaNo16} for the approximation of the optimal test space. Not least because of that, our proposed ansatz allows (especially in the parametric context) for a (very) easy implementation. However, in contrast to \cite{DaPlWe14,ZaNo16}, until now, we were not able to prove the convergence of the greedy algorithm proposed in this paper.  Note also that the RB approximation is no longer a linear combination of snapshots, but a linear combination of parameter-dependent applications of operators. Nevertheless the reproduction of snapshots is maintained. 
We finally note that our approach does not aim at obtaining an approximation that converges faster than the Kolmogorov $n$-width.

The remainder of this paper is organized as follows: In Section \ref{Sec:2} we present an optimally stable ultraweak variational formulation of first order linear transport equations, covering both time-independent and time-dependent operators. Section \ref{Sec:3} is devoted to the finite-dimensional, discrete case where we introduce an optimally stable Petrov-Galerkin method. Parametrized transport problems are considered in Section \ref{Sec:RBM} within the framework of the RB method. We describe the fairly easy com\-pu\-ta\-tion\-al realization of the new approach in Section \ref{Sec:5} and report on several numerical experiments in Section \ref{Sec:6}. Finally, we end with some conclusions in Section \ref{Sec:7}.

\section{An optimally stable ultraweak (space-time) formulation}\label{Sec:2}

In this section we present an ideally conditioned variational framework for  linear first order transport equations using results from \cite{DHSW2012} and \cite{AzePou96, Azerad1996}.
To that end, let $\Omega\subset\R^n$, $n\ge 1$, be a bounded polyhedral domain  with Lipschitz boundary, where we note that $\Omega$ may also be a space-time domain, as will be shown in Example \ref{subsec:example} at the end of this section. Moreover, $\vec{n}$ shall denote the outward normal of $\Gamma:=\partial\Omega$. Next, we introduce the advection field $\vec{b}(\cdot) \in C^{1}(\bar{\Omega})^{n}$ and the reaction coefficient $c(\cdot) \in C^{0}(\bar{\Omega})$, noting that for some statements the regularity assumption on $\vec{b}(\cdot)$ may be relaxed. We assume throughout this paper that
\begin{equation*}
c(z) - \tfrac{1}{2} \nabla \cdot \vec{b}(z) \geq 0 \quad  \text{for } z\in \Omega \enspace\text{almost everywhere}.
\end{equation*}
Then, we consider the first order transport equation 
\begin{equation}\label{eq:stationary}
	\begin{aligned}
	B_\circ u(z) 
		&:= \vec{b}(z)\cdot \nabla u(z) +c(z)u(z) = f_\circ(z),
		&&z\in\Omega,
		\\
	u(z)
		&=g(z),
		&& z\in\Gamma_{-}\equiv\Gamma_{\text{inflow}},
	\end{aligned}
\end{equation}
where $f_\circ \in C^{0}(\bar{\Omega})$, $g \in C^{0}(\overline{\Gamma_{-}})$, and $\Gamma_{\pm} := \{ z\in\partial \Omega:\, \vec{b}(z)\cdot \vec{n}(z) \gtrless 0\}$.  

For functions $v,w \in C^{0}(\bar{\Omega}) \cap C^{1}(\Omega)$ we obtain
\begin{equation*}
(B_\circ v, w)_{L_{2}(\Omega)} = (v, B_{\circ}^{*}w)_{L_{2}(\Omega)} + \int_{\Gamma_{-}} vw (\vec{b} \cdot \vec{n})\, ds + \int_{\Gamma_{+}} vw (\vec{b} \cdot \vec{n})\, ds, 
\end{equation*}
where $B_{\circ}^{*}w = - \vec{b} \cdot \nabla w + w(c - \nabla \cdot \vec{b})$ denotes the formal adjoint of $B_{\circ}$.\footnote{Considering \eqref{eq:stationary} with $g(z)\equiv 0$ and thus homogeneous Dirichlet boundary conditions we define  the formal adjoint $B_{\circ}^*$ of $B_{\circ}$ by 
$(B_{\circ} v,w)_{L_2(\Omega)} = (v, B_{\circ}^*w)_{L_2(\Omega)}$
	for all $v,w\in C^\infty_0(\Omega)$.} 
To account for the non-homogeneous boundary conditions we introduce as in \cite{DHSW2012} the spaces $C^{1}_{\Gamma_{\pm}}(\Omega) := \{ v \in C^{0}(\bar{\Omega}) \cap C^{1}(\Omega) \, : \, v|_{\Gamma_{\pm}} = 0 \}$ and obtain
\begin{equation*}
(B_{\circ} v,w)_{L_2(\Omega)} = (v, B_{\circ}^*w)_{L_2(\Omega)}, \quad v \in C^{1}_{\Gamma_{-}}(\Omega), w \in C^{1}_{\Gamma_{+}}(\Omega).
\end{equation*}
Thus, we may define the domain of $B_{\circ}^{*}$ as $\dom(B_{\circ}^{*}) = C^{1}_{\Gamma_{+}}(\Omega)$. 
For the derivation of a stable variational formulation we require as in \cite{DHSW2012} the following two assumptions:
\begin{assumption}
	\label{lemma:B12}
	We assume that the following conditions hold:
	\begin{compactenum}
		\item[(B1)] There exists a dense subspace $\dom(B_{\circ}^*) \subseteq L_2(\Omega)$ on which $B_{\circ}^*$ is injective.
		\item[(B2)] The \emph{range} $\ran(B_{\circ}^*):= \{  B_{\circ}^*v:\, v\in \dom(B_{\circ}^*)\}$ of $B_\circ^*$ is densely embedded in $L_{2}(\Omega)$.
	\end{compactenum}
\end{assumption}

We now give examples for conditions on the coefficient functions $\vec{b}$ and $c$ such that Assumption \ref{lemma:B12} holds true\footnote{We reuse condition (ii) of the corresponding Remark 2.2 in \cite{DHSW2012}. However, as can be seen from the counterexamples in Appendix \ref{sec:counterexamples}, Remark 2.2(i) in \cite{DHSW2012} is in general not sufficient for well-posedness. Therefore, we develop an alternative condition based on \cite{Azerad1996, AzePou96}.}.

\begin{proposition} \label{Prop:Poincare_B1_B2}
Let one of the following two conditions hold:
\begin{compactenum}
		\item[(i)] The flow associated with $\vec{b}$ is $\Omega$-filling, meaning that its trajectories starting from the inflow boundary do fill $\bar \Omega$ except perhaps for a set of measure zero in a finite bounded time $T$ (see \cite{AzePou96, Azerad1996})\footnote{For a more formal definition see Definition \ref{def:omega_filling_xi}.}. 
		\item[(ii)] There exists $\kappa > 0$ with $c - \frac{1}{2} \nabla \cdot \vec{b} \geq \kappa$ in $\Omega$ (see \cite[Remark 2.2(ii)]{DHSW2012}).
\end{compactenum}
Then, the operator $B^*_{\circ}$ satisfies (B1) and (B2). Moreover, we have the \emph{curved Poincar\'{e} inequality}
\begin{equation} \label{eq:curved_Poincare}
\|v\|_{L_2(\Omega)} \leq C \|B^*_{\circ} v\|_{L_2(\Omega)}, \quad v \in C^1_{\Gamma_+}(\Omega).
\end{equation}
In case of condition (i) the constant is $C = 2T$, in case (ii) $C = \frac{1}{\kappa}$.
\end{proposition}

\begin{proof}
See Appendix \ref{Proof:Poincare_B1_B2}.
\end{proof}

The following proposition gives a sufficient condition for $\vec{b}$ to have an $\Omega$-filling flow:

\begin{proposition}[{$\kern-3.5pt$\cite[Prop.\ 7]{Azerad1996}}] \label{Prop:Condition_Omega_Filling}
If $\vec{b} \in C^1(\bar\Omega)^n$ is bounded as well as its gradient in a neighborhood $V$ of $\bar\Omega$, if there are a unit vector $\vec{k}$, a number $\alpha > 0$ such that
\begin{equation}
\vec{b}(x) \cdot \vec{k} \geq \alpha \quad \forall x \in \bar \Omega,
\end{equation}
and if $\Omega$ is bounded in the $\vec{k}$ direction then the flow is $\Omega$-filling. 
\end{proposition}

We may now define as in \cite{DHSW2012}
\begin{equation*}
\|v\|_* := \|B_{\circ}^* v\|_{L_{2}(\Omega)}
\end{equation*}
and note that due to (B1) $\| \cdot \|_{*}$ is a norm on $\dom(B_{\circ}^*)$. 
With this framework at hand, we can define as in \cite{DHSW2012} the test space by
\[
	\cY := \clos_{\|\cdot\|_*} \{ \dom(B_{\circ}^*)\},
\]
which is a Hilbert space with inner product 
$(v,w)_\cY := (B^*v, B^*w)_{L_{2}(\Omega)}$ and induced norm $\| v\|_\cY := \|v\|_*,\ v,w\in\cY$.
Here, $B^* : \cY \to L_{2}(\Omega)$ denotes the continuous extension of $B_{\circ}^*$ from $\dom(B_{\circ}^*)$ to $\cY$. Then, we can define $B: L_{2}(\Omega)\to\cY'$ again by duality, i.e., $B := (B^*)^*$.  
The variational formulation of \eqref{eq:stationary} may then be based upon the bilinear form
\begin{equation}\label{eq:blf}
	b: L_{2}(\Omega)\times\cY\to\R,\qquad
	b(v,w) := ( v, B^* w )_{L_{2}(\Omega)} = \int_{\Omega} v ( -\vec{b} \cdot \nabla w + w(c - \nabla \cdot \vec{b})) \, dx.
\end{equation}
To incorporate the boundary conditions, we introduce as in \cite{DHSW2012} the weighted $L_{2}$-space $L_{2}(\Gamma_{-},|\vec{b} \cdot \vec{n}|)$ with norm $\|w\|_{L_{2}(\Gamma_{-},|\vec{b} \cdot \vec{n}|)} := (\int_{\Gamma_{-}} |w|^{2} |\vec{b} \cdot \vec{n}| \, ds )^{1/2}$ and show that functions in $\cY$ have a trace in $L_{2}(\Gamma_{-},|\vec{b} \cdot \vec{n}|)$:\footnote{Note, that due to a wrong estimate the constant given in the corresponding result \cite[Prop.\ 2.3]{DHSW2012} is generally not true. We therefore give a modified proof using \eqref{eq:curved_Poincare} for the estimate in question.}

\begin{proposition} \label{Prop:Trace}
Assume that one of the two conditions of Proposition \ref{Prop:Poincare_B1_B2} holds. Then, there exists a linear continuous mapping 
\begin{equation*}
\gamma_- : \mathcal{Y} \to L_2(\Gamma_-, |\vec{b} \cdot \vec{n}|),
\end{equation*}
such that
\begin{equation}
\|\gamma_-(v)\|_{L_2(\Gamma_-, |\vec{b}\cdot\vec{n}|)} \leq C_{tr} \|v\|_\mathcal{Y}, \quad v \in \mathcal{Y}.
\end{equation}
The constant is $C_{tr}= \sqrt{4T}$, or $C_{tr}=\sqrt{2\kappa^{-1}}$, respectively.
\end{proposition}
\begin{proof}
Integration by parts yields (see also \eqref{eq:est_aadj_v_v})
\begin{equation*}
(B^*_{\circ} v, v)_{L_2(\Omega)} = \int_{\Omega} v^2 (c- \tfrac{1}{2}\nabla \cdot b) dx 
- \tfrac{1}{2} \int_{\Gamma_-} v^2 \vec{b} \cdot \vec{n} ds. 
\end{equation*}
By using the general assumption $c - \frac{1}{2} \nabla \cdot \vec{b} \geq 0$ and $\vec{b} \cdot \vec{n} < 0$ on $\Gamma_-$, we have for $v \in C^1_{\Gamma_+}(\Omega)$
\begin{equation}
\int_{\Gamma_-} v^2 |\vec{b} \cdot \vec{n}|ds \leq 2|(B^*_{\circ}v, v)| \leq 2 \|v\|_{L_2(\Omega)} \|v\|_{\mathcal{Y}} \leq 2 C \|v\|_\mathcal{Y}^2,
\end{equation}
where we have used \eqref{eq:curved_Poincare} in the last estimate. The assertion for $v \in \mathcal{Y}$ follows by density.
\end{proof}

Next, we define for any $f_{\circ} \in L_{2}(\Omega)$ and $g \in L_{2}(\Gamma_{-},|\vec{b} \cdot \vec{n}|)$  a linear form $f \in \mathcal{Y}'$ as 
\begin{equation}\label{eq:rhs}
f(v) := (f_{\circ},v)_{L_{2}(\Omega)} + \int_{\Gamma_{-}} g \gamma_{-}(v)|\vec{b} \cdot \vec{n}| \, ds.  
\end{equation}

Then, we obtain the well-posedness of the variational formulation:

\begin{theorem}[{$\kern-3.5pt$\cite[Thm.\ 2.4]{DHSW2012}}]
	\label{Theo:Prop:DHSW2012}
	Assume that one of the two conditions in Proposition \ref{Prop:Poincare_B1_B2} is valid and $b$ and $f$ are defined as in \eqref{eq:blf} and \eqref{eq:rhs}, respectively. Then, there exists a unique $u\in L_{2}(\Omega)$ such that
	\begin{equation}\label{eq:weak_form}
		b(u,v) = f(v)\quad \forall v\in\cY,
	\end{equation}	
and the stability estimate 
$\| u \|_{L_{2}(\Omega)} \leq \|f \|_{\mathcal{Y}'}$ 
holds.  Moreover, 
	\begin{equation*}
	\sup_{w\in L_{2}(\Omega)} \sup_{v \in \cY} \frac{b(w,v)}{\|w\|_{L_{2}(\Omega)} \|v\|_\cY} 
	= \inf_{w\in L_{2}(\Omega)} \sup_{v \in \cY} \frac{b(w,v)}{\|w\|_{L_{2}(\Omega)} \|v\|_\cY} 
	= 1,
	\end{equation*}
	i.e., inf-sup and continuity constants are unity and, equivalently,
\begin{equation*}
\| B\|_{\cL(L_{2}(\Omega),\mathcal{Y}')} 
	= \|B^{*}\|_{\cL(\mathcal{Y},L_{2}(\Omega))} 
	= \|B^{-1}\|_{\cL(\mathcal{Y}',L_{2}(\Omega))} 
	= \|B^{-*}\|_{\cL(L_{2}(\Omega),\mathcal{Y})} = 1,
\end{equation*}	
where $B^{-*} := (B^*)^{-1} = (B^{-1})^*: L_{2}(\Omega)\to\cY$.
\end{theorem}
\begin{proof}
The proof follows the lines of the proof of \cite[Thm.\ 2.4]{DHSW2012} invoking Proposition \ref{Prop:Trace} instead of \cite[Prop.\ 2.3]{DHSW2012}.
\end{proof}

\begin{example}[Time-dependent linear transport equations]\label{subsec:example}
\hspace{-0.5em}The setting described in the beginning of this section includes both time-independent and time-dependent linear first order transport problems: As remarked in \cite{DHSW2012} we can consider time as an additional transport direction in the space-time domain, i.e., $z=(t,x)\in\Omega := (0,T) \times D = I\times D$, $n=1+d$, 
where $D\subset\R^d$ denotes the spatial domain.  Next, we define the space-time transport direction $\vec{b} := (1,\vec{b}_{x})^T \in C^{1}(\overline{I \times D})^{1+d}$, where $\vec{b}_{x}$ denotes the spatial advective field. Moreover, we introduce the space-time gradient operator as $\nabla := (\partial/\partial t, \nabla_{x})^T$, where $\nabla_{x}$ is the gradient on the spatial domain $D$. Accordingly, we set
$
	\Gamma_\pm := \{ (t,x)\in\Gamma:\, \vec{b}(t,x)\cdot \vec{n}(t,x) \gtrless 0\},
$
where $ \vec{n}(t,x)$ is again the outward normal of $\Gamma$. Then, we obtain exactly the form \eqref{eq:TE}, namely
\begin{equation*}
	B_\circ u := \vec{b}\cdot \nabla u + cu = f \quad\text{in } \Omega,
	\qquad
	u=g\quad\text{on }\Gamma_{-}.
\end{equation*}
For the space-time boundary, we have $\Gamma = \overline{\Gamma\ped{in} \cupdot \Gamma\ped{out} \cupdot \Gamma_{D}}$, where $\Gamma\ped{in} := \{ 0\}\times D$, $\Gamma\ped{out} := \{ T\}\times D$, $\Gamma_{D} := I\times \partial D$, along with its corresponding outward normals $\vec{n}\ped{in}:=(-1,0)^T$, $\vec{n}\ped{out}:=(1,0)^T$ and $\vec{n}_{D}:=(0,\vec{n}_{x})^T$, where $\vec{n}_{x}$ denotes the spatial outward normal (of $D$). Hence, $\vec{b}\cdot\vec{n}\ped{in}=-1$, $\vec{b}\cdot\vec{n}\ped{out}=1$, and $\vec{b}\cdot\vec{n}_{D}=\vec{b}_{x}\cdot\vec{n}_{x}$, so that $\Gamma_{-} = \overline{\Gamma\ped{in}\cupdot \Gamma_{D_{-}}}$, where $\Gamma_{D_{\pm}} = I\times\partial D_\pm$ and $\partial D_\pm := \{ x\in\partial D:\, \vec{b}_{x}(x)\cdot \vec{n}_{x}(x) \gtrless 0\}$. We emphasize that also non-homogeneous initial values are thus prescribed in an essential manner. 
Note that for the time-dependent case condition (i) of Proposition \ref{Prop:Poincare_B1_B2} is always fulfilled, which can be seen by taking $\vec{k} = (1, \vec{0})$ in Proposition \ref{Prop:Condition_Omega_Filling}, since $\vec{k} \cdot \vec{b} \equiv 1$ (cf.\ \cite{Azerad1996}). 
As an alternative to this realization of a space-time formulation, one could also treat spatial and temporal variables separately. Such a \emph{strong in time} variational formulation however results in a suboptimal inf-sup constant, for details see Appendix \ref{sec:strong_in_time}.
\end{example}

\section{An optimally stable Petrov-Galerkin method}\label{Sec:3}

In this section we introduce computationally feasible and optimally stable (conforming)  finite dimensional trial and test spaces $\cX^\delta\subset \cX=L_{2}(\Omega)$ and $\cY^\delta\subset\cY$ for the approximation of the solution of \eqref{eq:weak_form}. Here, we denote by $\delta$ a discretization parameter, where $\delta$ equals the mesh size $h$ for spatial problems and $\delta = (\Delta t, h)$ for time-dependent problems in space and time with a time step $\Delta t$\footnote{If we use a tensor product discretization in space, $\delta$ may also take the form $\delta = (h_1,\dots, h_d)$ or $\delta = (\Delta t, h_1, \dots, h_d)$, respectively.}. 
Then, the discrete counterpart of \eqref{eq:weak_form} reads
\begin{equation}\label{eq:disc_form}
	u^\delta\in\cX^\delta:
	\qquad b(u^\delta,v^\delta) = f(v^\delta)\quad \forall v^\delta\in\cY^\delta.
\end{equation}	
This latter equation admits a unique solution $u^\delta\in\cX^\delta$ provided that 
\begin{align}
	\gamma^\delta 
	:= \!\sup_{w^\delta\in\cX^\delta} 
		\sup_{v^\delta \in \cY^\delta} 
		\frac{b(w^\delta,v^\delta)}{\|w^\delta\|_{L_{2}(\Omega)}\, \|v^\delta\|_\cY} 
		\!<\!\infty, \,\,\,\,
		&
	\beta^\delta
	:= \!\inf_{w^\delta\in\cX^\delta} 
		\sup_{v^\delta \in \cY^\delta} 
		\frac{b(w^\delta,v^\delta)}{\|w^\delta\|_{L_{2}(\Omega)}\, \|v^\delta\|_\cY} 
		\!>\!0,
\end{align}
where we additionally require the existence of $\beta$ and $\gamma$ such that
\begin{equation*}
	\gamma^\delta \le \gamma<\infty,
	\quad
	\beta^\delta \ge \beta>0,
	\qquad
	\text{for all } \delta>0.
\end{equation*}
These constants for continuity $\gamma^\delta$ and stability (or inf-sup) $\beta^\delta$ also play a key role for the relation of the \emph{error}
$
	e^\delta := u-u^\delta
$
and the \emph{residual} $r^\delta\in\cY'$ defined as 
\begin{equation*}
	r^\delta(w) := f(w) - b(u^\delta,w) = b(e^\delta, w), \qquad w\in\cY,
\end{equation*}
as can be seen by the standard lines
\begin{align*}
\beta\, \| e^\delta\|_{L_{2}(\Omega)}
\le \sup_{w\in\cY} \frac{b(e^\delta, w)}{\| w\|_\cY}
= \sup_{w\in\cY} \frac{r^\delta(w)}{\| w\|_\cY} = \| r^\delta \|_{\cY'} 
\le \gamma \sup_{w\in\cY} \frac{\|e^\delta\|_{L_{2}(\Omega)}\, \| w\|_\cY}{\| w\|_\cY}
= \gamma\, \| e^\delta\|_{L_{2}(\Omega)}.
\end{align*}
In the optimal case, i.e., $\beta = \gamma =1$, error and residual coincide, i.e., $\| e^\delta\|_{L_{2}(\Omega)} = \| r^\delta \|_{\cY'}$. 
Moreover, we have the following C\'{e}a-type lemma \cite{XZ94}
\begin{equation*}
	\| e^\delta\|_{L_{2}(\Omega)}
	= \| u-u^\delta\|_{L_{2}(\Omega)}
	\le \frac{\gamma}{\beta} \inf_{v^\delta\in\cX^\delta} \| u-v^\delta\|_{L_{2}(\Omega)}
	= \frac{\gamma}{\beta} \,\sigma_{L_{2}(\Omega)} (u; \cX^\delta),
\end{equation*}
where 
$ 
	\sigma_{L_{2}(\Omega)} (u; \cX^\delta) 
	:= \inf_{v^\delta\in\cX^\delta} \| u-v^\delta\|_{L_{2}(\Omega)}
$ 
denotes the error of the best approximation to an element $u\in {L_{2}(\Omega)}$ in $\cX^\delta$ w.r.t.\ the $L_{2}(\Omega)$-norm. 
Since $u^\delta\in\cX^\delta$, it is trivially seen that
$
\sigma_\cX (u; \cX^\delta) 
\le \| u-u^\delta\|_{L_{2}(\Omega)}
= \| e^\delta\|_{L_{2}(\Omega)}$, 
so that in the optimal case $\beta = \gamma =1$ it holds that
\begin{equation}\label{eq:opt}
		\| r^\delta \|_{\cY'}
		= \| e^\delta\|_{L_{2}(\Omega)}
		= \sigma_{L_{2}(\Omega)} (u; \cX^\delta), 
\end{equation}
i.e., the numerical approximation is the best approximation.

\subsection{An optimally conditioned Petrov-Galerkin method}\label{SubSec:NewApproach}

To realize an optimally conditioned and thus optimally stable Petrov-Galerkin method, which is also computationally feasible, we suggest in this paper to first choose a conformal finite-dimensional test space $\cY^\delta\subset\cY$ and then to set 
\begin{equation}\label{eq:Opt2}
	\cX^\delta:= B^*(\cY^\delta) \subset L_{2}(\Omega).  
\end{equation}
For this pair of trial and test spaces we then obtain for every $w^{\delta} \in \mathcal{X}^{\delta}$ that
\begin{equation}\label{eq:inf-sup is one}
\sup_{v^{\delta} \in \mathcal{Y}^{\delta}} \! \frac{b(w^{\delta},v^{\delta})}{\|w^{\delta}\|_{L_{2}(\Omega)} \|v^{\delta}\|_{\mathcal{Y}}} \! = \! \frac{b(w^{\delta},B^{-*}w^{\delta})}{\|w^{\delta}\|_{L_{2}(\Omega)} \|B^{-*}w^{\delta}\|_{\mathcal{Y}}} \! = \! \frac{(w^{\delta},B^{*}\!B^{-*}w^{\delta})_{L_{2}(\Omega)}}{\|w^{\delta}\|_{L_{2}(\Omega)} \|B^{*}\!B^{-*}w^{\delta}\|_{L_{2}(\Omega)}}  \equiv  1.
\end{equation}
Here, we have exploited the fact that for all $w^{\delta} \in \mathcal{X}^{\delta}$ for the supremizer $s^{\delta}_{w^{\delta}} \in \mathcal{Y}^{\delta}$, defined as the solution of
$(s^{\delta}_{w^{\delta}}, v^{\delta})_{\mathcal{Y}} = b(w^{\delta}, v^{\delta})$ 
for all $v ^{\delta} \in \mathcal{Y}^{\delta}$, 
we have $s^{\delta}_{w^{\delta}} = B^{-*}w^{\delta}$ as $B^{*}$ is boundedly invertible. From \eqref{eq:inf-sup is one} we may thus conclude that indeed
\begin{equation}\label{eq:inf sup 1 part 2}
\beta^\delta = \gamma^\delta = 1
\end{equation}
and the proposed method is optimally stable. We note that the same approach is investigated in parallel for the wave equation in \cite{RB:Wave}.

Moreover, we emphasize that the suggested approach is computationally feasible since $B^*$ is a differential operator which can easily be applied -- as long as the test space is formed by `easy' functions such as splines as in the case of finite elements (FE). Additionally, for our choice of test and trial space we may reformulate the discrete problem \eqref{eq:disc_form} as follows: Thanks to the definition of the trial space $\cX^\delta$ in \eqref{eq:Opt2} there exists  for all $v^{\delta} \in \mathcal{X}^{\delta}$ a unique $w^\delta \in \cY^\delta$ such that $v^\delta = B^* w^\delta$. Therefore, the problem \eqref{eq:disc_form} is equivalent to the problem
 \begin{equation} \label{eq:discrete_reformulated}
	w^\delta \in \cY^\delta: \qquad
	a( w^\delta, v^\delta)
	:= (B^* w^\delta, B^* v^\delta)_{L_2(\Omega)} 
	= f(v^\delta) \quad \forall v^\delta \in \cY^\delta,
\end{equation}
which obviously is a symmetric and coercive problem, the normal equations, or a least-squares problem. Thus, problem \eqref{eq:discrete_reformulated} is well-posed and we identify the solution of \eqref{eq:disc_form} as $u^\delta := B^*  w^\delta$.
This reformulation will also be used for the implementation of the framework. From \eqref{eq:discrete_reformulated} we see that for the setup of the linear system for $w^\delta$ the precise knowledge of the basis of $\cXd = B^* \cYd$ is not needed -- only for the pointwise evaluation of $u^\delta$ when e.g.\ visualizing the solution. 
For further details on the computational realization we refer to Section \ref{Sec:5}.

Thanks to \eqref{eq:inf sup 1 part 2} we are moreover in the optimal case described in the beginning of this section and the numerical approximation $u^\delta\in\cX^\delta$ is thus the best approximation of $u\in L_{2}(\Omega)$ for our suggested choice of trial and test space. Hence, we obtain $\|e^\delta\|_{L_2(\Omega)}=\sigma_{L_2}(u,\cX^\delta) = \| r^\delta\|_{\cY'}$. 
Due to \eqref{eq:Opt2} we have that for any $w^\delta\in\cX^\delta$ there exists a unique $v^\delta\in\cY^\delta$ with $B^*v^\delta=w^\delta$. In view of (B1) in Assumption \ref{lemma:B12}, there also exists a unique $v\in\cY$ such that $B ^*v=u$, namely $v^*=B^{-*}u$. 
Therefore, 
\begin{align}
\nonumber	\|e^\delta\|_{L_2(\Omega)} 
	&= \sigma_{L_2}(u,\cX^\delta) 
	= \inf_{w^\delta\in\cX^\delta} \| u-w^\delta\|_{L_2(\Omega)}
	= \inf_{v^\delta\in\cY^\delta} \| B^*v-B^*v^\delta\|_{L_2(\Omega)} \\[-1.5ex]
	&\label{eq:convergence}\\[-1.5ex]
\nonumber	&= \inf_{v^\delta\in\cY^\delta} \| v- v^\delta\|_{\cY} 
	= \sigma_{\cY}(B^{-*}u,\cY^\delta).
\end{align}
We may thus also infer from \eqref{eq:convergence}	the (strong) convergence of the approximation $u^{\delta}$ to $u$ in $L_{2}(\Omega)$ provided that $\inf_{v^\delta\in\cY^\delta} \| v- v^\delta\|_{\cY}$ converges to $0$ as $\delta \rightarrow 0$. Note, that the latter can be ensured by choosing an appropriate test space $\mathcal{Y}^{\delta}$ as say a standard FE space.

We finally remark that in standard FE methods the error analysis is usually done in two steps: (1) relation of the error to the best approximation by a C\'ea-type lemma; (2) proving an asymptotic rate of convergence e.g.\ by using a Cl\'ement-type interpolation operator. As seen above, (1) also holds for our new trial spaces -- in a non-standard norm, however. Regarding the second step (2) there is hope that it might maybe be possible to derive convergence rates via the term $\inf_{v^\delta\in\cY^\delta} \| v- v^\delta\|_{\cY}$ (see \eqref{eq:convergence}) and mapping properties of the operator $B$. This is however beyond the scope of the present paper and the subject of future work. 
Here, we will hence investigate the rate of convergence in numerical experiments in Section \ref{Sec:6}.

\begin{example}[Illustration of trial space] \label{Ex_Illustration_Trial}
We illustrate the trial space $\cX^{\delta}$ as defined in \eqref{eq:Opt2} for a very simple, one-dimensional problem. In detail, we consider $\Omega:=(0,1)$, a constant transport term $b>0$, and a variable reaction coefficient $c \in C^0([0,1])$; that means $B_\circ u(x) := b\, u'(x) + c(x)\, u(x)$, $x\in\Omega$, as well as $u(0)=g$ on $\Gamma_-=\{0\}$. 
We get $B_\circ^* v(x) := -b\, v'(x) + c(x)\, v(x)$. According to our proposed approach, we start by defining a test space $\cY^h$. To this end, let $n_h\in\N$ and $h:=\frac1{n_h}$, $I_i:= [(i-1)h, ih)\cap\bar\Omega$, $i=1,\ldots, n_h$, $I_{0}:=\emptyset$. We use standard piecewise linear FE, i.e.,
\begin{equation*}
	\eta_i (x) := 
		\begin{cases} 	
			\frac{x}{h} + 1-i, &\mbox{if } x\in I_{i-1}, \\ 
			-\frac{x}{h} + 1+i, &\mbox{if } x\in I_{i},\\ 
			\qquad 0, &\mbox{else,} 
		\end{cases} 
\end{equation*}
for $i=1,\ldots , n_h$ and define $\cY^h:=\spanlin\{ \eta_1, \ldots , \eta_{n_h}\}$.  Then, we construct the optimal trial space in the above sense by $\cX^h:=\spanlin\{ \xi_1, \ldots , \xi_{n_h}\}$, where we set
\begin{equation*}
	\xi_i(x) := B^*\eta_i (x) = -b\, \eta_i'(x) + c(x)\, \eta_i(x)
	= 
		\begin{cases} 	
			-\frac{b}{h} + c(x) ( \frac{x}{h} +1-i ), &\mbox{if } x\in I_{i-1}, \\ 
			\frac{b}{h} + c(x) ( -\frac{x}{h} + 1+i ), &\mbox{if } x\in I_{i},\\ 
			\qquad 0, &\mbox{else,} 
		\end{cases} 
\end{equation*}
for $i=1,\ldots , n_h$. Note, that for the special case of constant reaction $c(x)\equiv c$, the functions $\xi_i$ are piecewise linear and discontinuous, see Figure \ref{Fig:1D_Basis_Functions}.

\begin{figure}[tb]
	\pgfplotsset{title style={at={(0.5,0.85)}}}
	\begin{center}
		{\footnotesize 
			\begin{tikzpicture}\hypersetup{hidelinks}
			\begin{axis}[width=6.5cm, height=3.25cm, title={Basis of $\mathcal{Y}^h$}]
			\addplot[red,thick,domain=0:0.25,samples=10,]{-4*x + 1};
			\addplot[blue,thick,domain=0:0.25,samples=10,]{4*x };
			\addplot[blue,thick,domain=0.25:0.5,samples=10,]{-4*x + 2};
			\addplot[green,thick,domain=0.25:0.5,samples=10,]{4*x - 1};
			\addplot[green,thick,domain=0.5:0.75,samples=10,]{-4*x + 3};
			\addplot[orange,thick,domain=0.5:0.75,samples=10,]{4*x - 2};
			\addplot[orange,thick,domain=0.75:1,samples=10,]{-4*x + 4};
			\end{axis}
			\end{tikzpicture}
			\begin{tikzpicture}\hypersetup{hidelinks}
			\begin{axis}[width=6.5cm, height=3.25cm, title={Basis of $\mathcal{X}^h = B^*\mathcal{Y}^h$}]
			\addplot[red,thick,domain=0:0.25,samples=10,]{4 - 8*x + 2};
			\addplot[blue,thick,domain=0:0.25,samples=10,]{-4 + 8*x };
			\draw[blue,thick] (axis cs:0.25,-2) -- (axis cs:0.25,6);
			\addplot[blue,thick,domain=0.25:0.5,samples=10,]{4 -8*x + 4};
			\addplot[green,thick,domain=0.25:0.5,samples=10,]{-4 + 8*x - 2};
			\draw[green,thick] (axis cs:0.5,-2) -- (axis cs:0.5,6);
			\addplot[green,thick,domain=0.5:0.75,samples=10,]{4 -8*x + 6};
			\addplot[orange,thick,domain=0.5:0.75,samples=10,]{-4 + 8*x - 4};
			\draw[orange,thick] (axis cs:0.75,-2) -- (axis cs:0.75,6);
			\addplot[orange,thick,domain=0.75:1,samples=10,]{4 -8*x + 8};
			\end{axis}
			\end{tikzpicture}
		}
		\end{center}
		\caption{Basis functions of $\cY^h$ and $\cX^h$ for $h=\frac{1}{4}$, $b\equiv1$, $c\equiv2$. \label{Fig:1D_Basis_Functions}}
\vspace{-1.5em}
\end{figure}
\end{example}

\subsection{Nonphysical restrictions at the boundary}  \label{Subsec:Problem_tensor}

From a computational perspective it is appealing to use discrete spaces that are tensor products of one-dimensional spaces; for details see Section \ref{Sec:5}.
However, this choice may result in nonphysical restrictions of functions in the trial space on certain parts of the outflow boundary. 

To illustrate this, consider $\Omega = (0,1)^2$ and let $\vec{b} \equiv (b_1,b_2)^T \in \R^2$, $c \in \R$ with $b_1,b_2>0$, such that we have for the inflow boundary $\Gamma_- = ( \{0\} \times (0,1) ) \cup ( (0,1) \times \{0\} )$ and thus for the outflow boundary $\Gamma_+ = ( \{1\} \times (0,1) ) \cup ( (0,1) \times \{1\} )$. Let $\cY^h_{\oneD}$ be a univariate finite dimensional space with $\cY^h_{\oneD} = \spanlin \{\phi_1, \dots, \phi_{n_h}\} \subset H^1_{(1)}(0,1) := \{ v\in H^1(0,1): v(1)=0\}$.
Next, we define the discrete test space on $\Omega = (0,1)^2$ as the tensor product space
\begin{equation*}
\cY^\delta 
:= \cY^h_{\oneD} \otimes \cY^h_{\oneD}
= \spanlin \{ \phi_i\otimes \phi_j:\, 1\le  i,j\le n_h \},
\qquad \delta = (h,h).
\end{equation*}
Then, the optimal trial functions are given for $i,j, = 1, \ldots , n_h$, by 
\begin{equation*}
\psi_{i,j} := B^* (\phi_i \otimes\phi_j)
= -b_1 (\phi_i'\otimes\phi_j) -b_2 (\phi_i\otimes\phi_j') + c (\phi_i\otimes\phi_j)
\end{equation*}
and we set $\cX^\delta := \spanlin \{ \psi_{i,j}:\, 1\le  i,j\le {n_h}\}$. 
However, this simple tensor product ansatz results in $\psi_{i,j}(1,1)=0$ for all $i$ and $j$, i.e., any numerical approximation would vanish at the right upper corner $(1,1)\in\overline{\Omega}$. Needless to say that this is a nonphysical restriction at the boundary, even though point values do not matter for an $L_2$-approximation. It is obvious that the 2D-case is only the simplest one in which this effect appears. In fact, in a general $d$D-situation ($d\ge 2$), we would obtain that optimal trial functions constructed as the $B^*$-image of tensor products would vanish on $(d-2)$-dimensional sets along the boundary of $\overline{\Omega}$, leading to nonphysical boundary values.
To reduce the impact of this effect, we suggest to consider an additional ``layer'' around the computational domain by defining a tube of width $\alpha>0$ around $\Gamma_+$ by 
\begin{equation} \label{eq:extended_domain}
\Omega_+(\alpha) := \{ x\in\R^n\setminus\Omega:\, \exists y\in\Gamma_+: \| x-y\|_\infty< \alpha\},
\qquad
\Omega(\alpha):=\Omega\cup\Omega_+(\alpha).
\end{equation}
Then, we solve the original transport problem on the extended domain $\Omega(\alpha)$ using the associated pair of optimal trial and test spaces. As a result the trial functions vanish on the exterior boundary of $\Omega_+(\alpha)$, but not on $\partial\Omega$. From a numerical perspective, by choosing $\alpha = mh$ for a (small) $m \in \N$ and the mesh size $h$, this adds $m$ layers of grid cells and thus $\mathcal{O}(n_h^{d-1})$ degrees of freedom. 
On the larger domain $\Omega_+(\alpha)$, the numerical solution remains to be a best-approximation in the enlarged trial space. Due to the larger dimension, this is no longer true w.r.t.\ the original domain $\Omega$. However, note, that the additional unknowns are only $(d-1)$-dimensional. 
We will numerically investigate this effect in Section \ref{Sec:6}.

\subsection{Post-processing} \label{Subsec:Post-processing}

As already mentioned, we are particularly interested in using our framework for problems with non-regular solutions $u \in L_2(\Omega)$, which especially includes jump discontinuities that are transported through the domain. 
However, it is well known that (piecewise) polynomial $L_2$-approximations of such discontinuities result -- especially for higher polynomial orders -- in overshoots, the so-called Gibbs phenomenon. There are many works concerning post-processing techniques to mitigate such effects, see for instance \cite{Shu2014} and the references therein.

Within the scope of this paper, we restrict ourselves to a rather simple post-processing procedure aimed at limiting the solution near jump discontinuities. 
Let  $\cYd \subset \cY$ be a conforming FE test space on $\Omega\subset\R^n$ corresponding to a partition $\mathcal{T}_\delta = \{K_i\}_{i=1}^{n_{\mathcal{T}_\delta}}$ of $\Omega = \bigcup_{i=1}^{n_{\mathcal{T}_\delta}} K_i$ with polynomial order $p \geq 2$:
\begin{equation*}
\cYd := \{ v \in C^0(\Omega): v|_K \in \mathbb{P}^p(K)\, \forall K \in \mathcal{T}_\delta, v|_{\Gamma_+}=0 \} \subset \cY.
\end{equation*}
If $w^\delta \in \cYd$ denotes the solution to \eqref{eq:discrete_reformulated}, the solution $u^\delta \in \cXd = B^* \cYd$ to \eqref{eq:disc_form} reads
$
u^\delta = B^* w^\delta =  - \sum_{i=1}^n b_i \partial_{x_i} w^\delta + (c - \nabla \cdot \vec{b}) w^\delta.
$
Since $w^\delta \in \cYd$ is a FE function, the partial derivatives $\partial_{x_i} w^\delta, i=1,\dots, n$ contain discontinuities across the cell boundaries, such that limiting these terms has the potential to mitigate overshoot effects. For all $K \in \mathcal{T}_\delta$, we have $\partial_{x_i} w^\delta|_K \in \mathbb{P}^p(K)$. Based upon this, we define  
\begin{equation*}
	\widetilde{\partial_{x_i} w^\delta} \in L_2(\Omega) 
	\quad \text{by} \quad
	\widetilde{\partial_{x_i} w^\delta}|_K := P_{\mathbb{P}^{(p-1)}(K)} \partial_{x_i} w^\delta|_K \quad \forall K \in \mathcal{T}_\delta,
\end{equation*} 
where $P_{\mathbb{P}^{(p-1)}(K)}$ is the $L_2$-orthogonal projection onto the polynomials of order at most $(p-1)$ on $K$. We then define the post-processed solution to \eqref{eq:disc_form} as 
\begin{equation*}
\tilde{u}^\delta := - \sum_{i=1}^n b_i \widetilde{\partial_{x_i} w^\delta} + (c - \nabla \cdot \vec{b}) w^\delta.
\end{equation*}
As a first attempt, one may perform the element-wise $L_2$-projection on all grid cells. However, for many problems it might be better (or even necessary) to choose a set of grid cells $\mathcal{T}_\delta^{\text{jump}} \subset \mathcal{T}_\delta$ that contains all cells where overshoots due to the jumps indeed occur, and only perform the post-processing for the cells $K \in \mathcal{T}_\delta^{\text{jump}}$. For methods that are able to detect such cells we refer to \cite{QiuShu2005}.

Due to the construction of the post-processed solution independent from the trial space $\mathcal{X}^\delta$, it is not clear whether the post-processed solution shows the same convergence rate as the standard solution. We will investigate the convergence behavior in numerical examples in Section \ref{Sec:6}.
We will test this approach 
for piecewise constant solutions $u$ with jump discontinuities. For more complex problems, perhaps other, more sophisticated methods from the literature have to be used.

\section{The reduced basis method for parametrized transport problems}\label{Sec:RBM}

In this section we generalize the above setting to problems depending on a parameter and apply the reduced basis method for that purpose, \cite{QuMaNe16,HeRoSt16,Haasdonk:RB}.

\subsection{Parametrized transport problem} \label{Subsec:Param_Probl}
We consider a parametrized problem based upon a compact set of parameters $\cP \subset \R^p$. In analogy to the above framework we define the domain $\Omega$ and the now possibly parameter-dependent quantities $\vec{b}_\mu \in C^1(\bar{\Omega})^n$ and $c_\mu \in C^0(\bar\Omega)$ with $c_\mu - \frac{1}{2} \nabla \cdot \vec{b}_\mu \geq 0$ for all $\mu \in \cP$. For all $\mu \in \cP$ we define  $f_{\mu; \circ} \in C^0(\bar\Omega)$ and $g_\mu \in C^0(\bar{\Gamma}_-)$. Then we consider the parametric problem of finding $u_\mu:\Omega\to\R$ such that
\begin{equation*}\label{eq:parametric}
	\begin{aligned}
	B_{\mu;\circ} u_\mu(z) 
		&:= \vec{b}_\mu(z)\cdot \nabla u_\mu(z) + c_\mu(z) u_\mu(z) = f_{\mu; \circ}(z),
		&&z\in\Omega,
		\\
	u_\mu(z)
		&=g_\mu(z),
		&& z\in\Gamma_{-}.
	\end{aligned}
\end{equation*}

\begin{assumption}\label{Ass:ParBoundaries}
We assume that  $\Omega$, $\cP$ and $\vec{b}_\mu$ are chosen such that the in- and outflow boundaries $\Gamma_{\pm} := \{ z\in\partial \Omega:\, \vec{b}_\mu(z)\cdot \vec{n}(z) \gtrless 0\}$ are \emph{parameter-independent}. 
\end{assumption}

\begin{remark}
As we shall see below, Assumption \ref{Ass:ParBoundaries} is a direct consequence of a necessary density assumption to be formulated below. However, as stated in \cite{DaPlWe14}, for parameter-dependent $\Gamma_\pm(\mu)$ and a polyhedral domain $\Omega$, it is always possible to decompose $\cP$ into a finite number of subsets $\cP_m$, $m=1,\dots, M,$ with fixed parameter-independent corresponding in- and outflow boundaries. Hence, one considers $M$ sub-problems on $\cP_m$, $m=1,\dots,M$, with separate reduced models.  
Moreover, one could also consider parameter-dependent $\Omega_\mu$, $\Gamma_{\pm,\mu}$ that can be transformed onto a parameter-independent reference domain $\Omega$ with fixed in- and outflow boundaries by varying the data. 
\end{remark}

Next, we require Assumption \ref{lemma:B12} for the formal adjoint $B^*_{\mu;\circ}$ for all $\mu \in \cP$ such that we can apply the above framework separately for all $\mu \in \cP$ in order to define the test space $\cY_\mu$ with parameter-dependent norm $\|v\|_{\mathcal{Y}_{\mu}} := \|B^*_\mu v \|_{L_2(\Omega)}$ as well as the extended operators $B_\mu : L_2(\Omega) \to \cY'_\mu$ and $B^*_\mu : \cY_\mu \to L_2(\Omega)$.
Hence, we
aim at determining solutions $u_\mu \in L_2(\Omega)$ such that
\begin{equation}\label{eq:parametrized weak form}
b_\mu(u_\mu, v) := (u_\mu, B^*_\mu v)_{L_2(\Omega)} = f_\mu(v) \quad \forall v \in \cY_\mu.
\end{equation}
Note, that thanks to the definition of $\mathcal{Y}_{\mu}$ we have $\| B_{\mu}^{-*}\|_{L(L_{2}(\Omega),\mathcal{Y}_{\mu})}=1$ and therefore
\begin{equation}\label{eq:stability of solution}
\| u_{\mu} \|_{L_{2}(\Omega)} \leq \|f_{\mu}\|_{\mathcal{Y}_{\mu}'}.
\end{equation}

We mention that the norms $\|\cdot\|_{\cY_\mu}$ cannot be expected to be pairwise equivalent for different $\mu \in \cP$, which means that even the sets of two test spaces $\cY_{\mu_1}$, $\cY_{\mu_2}, \mu_1 \neq \mu_2$, can differ. Therefore, we define as in \cite{DaPlWe14} the parameter-independent test space
\begin{equation*}
\bar \cY := \bigcap_{\mu \in \cP} \cY_{\mu},
\end{equation*}
where we assume that $\bar{\mathcal{Y}}$ is dense in $\mathcal{Y}_{\mu}$ for all $\mu \in \mathcal{P}$.%
\footnote{This assumption, which is required for instance for Lemma \ref{lemma:compactness solution set}, automatically implies that $\Gamma_\pm$ are parameter-independent (Assumption \ref{Ass:ParBoundaries}), since a homogeneous Dirichlet boundary condition on $\Gamma_+$ is included in the test spaces.} Thanks to the compactness of $\mathcal{P}$ we may equip $\bar{\mathcal{Y}}$ with the norm 
$ 
\|v\|_{\bar \cY} := \sup_{\mu \in \cP} \|v\|_{\mathcal{Y}_{\mu}}.
$ 
The above theory of optimal trial and test spaces as well as well-posedness immediately extends to the parameter-dependent case in an obvious manner.

As usual, we assume that $B^*_\mu$ and $f_\mu$ are affine w.r.t.\ the parameter. In detail, we assume that there exist functions $\theta^{q}_b \in C^{0}(\bar{\mathcal{P}})$ for $q=1,\dots,Q_b$ and $\theta^{q}_f \in C^{0}(\bar{\mathcal{P}})$ for $q_f = 1,\dots, Q_f$
and $\mu$-independent operators $(B^{q})^{*} \in L(\bar{\cY},L_2(\Omega)), q=1,\dots,Q_b$  and linear functionals $f^{q} \in \bar{\cY}', q_f = 1,\dots, Q_f$ such that for all $\mu \in \cP$ we have 
\begin{equation}\label{eq:affine decomposition}
B^*_\mu = \sum_{q=1}^{Q_b} \theta^q_b(\mu)\, (B^q)^* \in L(\cY_\mu, L_2(\Omega)), \quad 
f_\mu = \sum_{q=1}^{Q_f} \theta^q_f(\mu)\, f^q  \in \cY_\mu'.
\end{equation}

\begin{lemma}\label{lemma:compactness solution set}
Under the above assumptions, the set  $\mathcal{M}:= \{ u_{\mu} \, \text{solves} \, \eqref{eq:parametrized weak form},\enspace \mu \in \mathcal{P}\}$ of solutions is a compact subset of $L_{2}(\Omega)$.
\end{lemma}
\begin{proof} (Sketch)
The main idea of the proof is to exploit the continuity of the mappings $\mu \mapsto B_{\mu}^{*}$ and $\mu \mapsto f_{\mu}$  to show that $\tilde{u}$ satisfies \eqref{eq:parametrized weak form} for some $\mu$ for all $v \in \bar{\mathcal{Y}}$ and subsequently use a density argument, see Appendix \ref{proof:compactness solution set} for details.
\end{proof}

\subsection{Discretization} \label{Subsec:Param_Discr}
For the discretization of the parametric problem, we introduce a parameter-independent discrete space $\cY^\delta \subset \bar \cY$. Next, for fixed $\mu \in \cP$ we define the discrete test space and the corresponding trial space as
\begin{equation*}
	\cY^\delta_\mu := \clos_{\|\cdot\|_{\mathcal{Y}_{\mu}}}(\cY^\delta)\subset \cY_\mu,
	\qquad
	\cX^\delta_\mu := B^*_\mu (\cY^\delta) \subset L_2(\Omega).
\end{equation*}
Note, that for different $\mu\in\cP$, the spaces $\cX_\mu^\delta$ \emph{differ as sets} but have the \emph{common norm} $\|\cdot\|_{L_2(\Omega)}$, whereas the spaces $\cY^\delta_\mu$ consist of the \emph{common set} $\cY^\delta$ with \emph{different norms} $\|\cdot\|_{\mathcal{Y}_{\mu}}$.
With the same reasoning as for the non-parametric case (see \eqref{eq:inf-sup is one}), we have an optimal discrete inf-sup constant for all $\mu \in \cP$, i.e., $\beta^\delta_\mu := \inf\limits_{w^\delta\in \cX^\delta_\mu} \sup\limits_{v^\delta \in \cY^\delta_\mu} \frac{b_\mu(w^\delta,v^\delta)}{\|w^\delta\|_{L_2(\Omega)} \|v^\delta\|_{\cY_\mu}} = 1$. 
The discrete solution $u^\delta_\mu \in \cX^\delta_\mu$ is then defined via
\begin{equation} \label{eq:discrete_param}
u^\delta_\mu \in \cX^\delta_\mu: \quad 
	b_\mu(u^\delta_\mu, v^\delta) 
	= (u^\delta_\mu, B^*_\mu v^\delta)_{L_2(\Omega)} 
	= f_\mu(v^\delta) 
	\quad \forall v^\delta \in \cY^\delta_\mu.
\end{equation}
As in \S \ref{SubSec:NewApproach} we observe that problem \eqref{eq:discrete_param} is equivalent to the problem 
\begin{equation}
\label{eq:discrete_param_reform}
w^\delta_\mu \in \cY^\delta_\mu: \qquad
a_\mu( w^\delta_\mu, v^\delta)
:= (B^*_\mu w^\delta_\mu, B^*_\mu v^\delta)_{L_2(\Omega)} 
= f_\mu(v^\delta) \quad \forall v^\delta \in \cY^\delta_\mu
\end{equation}
and we may thus solve \eqref{eq:discrete_param_reform} and identify the solution of \eqref{eq:discrete_param} as $u^\delta_\mu := B^*_\mu  w^\delta_\mu$.

\begin{remark}
Since for all $\mu \in \cP$ we have 
$\cX^\delta_\mu = B^*_\mu (\cYd) = \sum_{q=1}^{Q_b} \theta_b^q(\mu) (B^q)^* (\cYd)$, 
there holds 
$\cX^\delta_\mu \subset \widehat{\cXd} := (B^1)^* (\cYd) + \cdots + (B^{Q_b})^* (\cYd) \subset L_2(\Omega)$, 
which means that the trial spaces for all $\mu \in \cP$ are contained in a common discrete space with dimension $\dim \widehat{\cXd} \leq Q_b \cdot \dim \cYd$.
\end{remark}

\begin{corollary}\label{coro:compactness discrete solution set}
Under the above assumptions the discrete solution set \\$\mathcal{M}^{\delta}:=\{ u^{\delta}_{\mu} \enspace \text{solves} \enspace \eqref{eq:discrete_param}, \enspace  \mu \in \mathcal{P}\} \subset \widehat{\mathcal{X}^{\delta}}$ is a compact subset of $\widehat{\mathcal{X}^{\delta}}$.
\end{corollary}
\begin{proof}
The proof can be done completely analogous to the continuous setting exploiting that $\widehat{\mathcal{X}^{\delta}}$ is a Hilbert space equipped with the $L_{2}$-inner product.
\end{proof}

\subsection{Reduced scheme}\label{SubSec:4.3}

We assume that we have at our disposal a reduced test space $Y^N \subset \cY^\delta$ with dimension $N \in \N$\footnote{In order to have a clear distinction between high- and low-dimensional spaces, we use calligraphic letters for the high-dimensional and normal symbols for the reduced spaces.} constructed for instance via a greedy algorithm (see \S \ref{Subsec:Basis_generation}).
Then, for each $\mu \in \cP$ we introduce the reduced discretization with test space $Y^N_\mu := \clos_{\|\cdot\|_{\mathcal{Y}_{\mu}}}(Y^N) \subset \cY^\delta_\mu$ and trial space $X^N_\mu := B^*_\mu (Y^N_\mu) \subset \cX^\delta_\mu$. The reduced problem then reads
\begin{equation}\label{eq:redform}
	u^N_\mu \in X^N_\mu:
	\qquad
	b_\mu(u^N_\mu, v^N) 
	= (u^N_\mu, B^*_\mu v^N)_{L_2(\Omega)} = f_\mu(v^N) 
	\quad \forall v^N \in Y^N_\mu.
\end{equation}
As in the high-dimensional case discussed in \S \ref{Subsec:Param_Discr}, these pairs of spaces yield optimal inf-sup constants 
\begin{equation*}
\beta^N_\mu := \inf_{w^N \in X^N_\mu} \sup_{v^N \in Y^N_\mu} \frac{b_\mu(w^N, v^N)}{\|w^N\|_{L_2(\Omega)} \|v^N\|_{\cY_\mu}} = 1 \quad \text{for all } \mu \in \cP.
\end{equation*}
Hence, regardless of the choice of the `initial' reduced test space $Y^N$ we get a perfectly stable numerical scheme without the need to stabilize. Note, that this is a major difference to the related work\cite{DaPlWe14}, where, due to a different strategy in finding discrete spaces, a stabilization procedure is necessary. 
Using the least-squares-type reformulation \eqref{eq:discrete_param_reform}, we can (similarly to \eqref{eq:discrete_reformulated}) first compute $w^N_\mu \in Y^N_\mu$ such that
\begin{equation} \label{eq:reduced_param_reform}
	a_\mu(w^N_\mu, v^N) 
	=  (B^*_\mu w^N_\mu, B^*_\mu v^N)_{L_2(\Omega)} 
	= f_\mu(v^N) \quad \forall v^N \in Y^N_\mu,
\end{equation}
and then set $u^N_\mu := B^*_{\mu} w^N_\mu$ as the solution of \eqref{eq:redform}.

\subsubsection*{Offline-/Online-Decomposition} 
By employing the assumed affine parameter dependence of $B^*_\mu$ and $f_\mu$, the computation of $u^N_\mu$ can be decomposed efficiently in an offline and online stage: Let $\{ v^N_i:\, i=1,\dots,N\}$ be a basis of the parameter-independent test space $Y^N$. In the offline stage, we precompute and store the following parameter-independent quantities:
\begin{align*}
	b_{q,i} &:= (B^q)^* v^N_i, 
		&& \text{for  } q=1,\dots,Q_b, \, i=1,\dots, N, \\
	 A_{q_1, q_2; i,j} &:= ( b_{q_1,i}, b_{q_2,j} )_{L_2(\Omega)}, 
		&& \text{for  }q_1, q_2=1,\dots,Q_b, \, i,j=1,\dots, N, \\
	f_{q,i} &:= f^q(v^N_i), 
		&& \text{for  } q=1,\dots, Q_f, \, i=1,\dots, N.
\end{align*}
In the online stage, given a new parameter $\mu \in \cP$, we assemble for all  $i,j=1,\dots, N$ 
\begin{align*}
	({\mathbf{A}}^N_\mu)_{i,j}
		&:= (B^*_\mu v^N_i, B^*_\mu v^N_j )_{L_2(\Omega)} 
		= \sum_{q_1=1}^{Q_b} \sum_{q_2=1}^{Q_b} 
		\theta^{q_1}_b(\mu) \theta^{q_2}_b(\mu) 
		{A}_{q_1, q_2; i,j} ,  \\
	(\mathbf{f}^N_\mu)_i
		&:= f_\mu( v^N_i) 
		= \sum_{q=1}^{Q_f} \theta^q_f(\mu) f_{q,i}.
\end{align*}
Next, we compute $w^N_\mu = \sum_{i=1}^N w_i(\mu)\, v^N_i \in Y^N$ as in \eqref{eq:reduced_param_reform} by solving the linear system ${\mathbf{A}}^N_\mu \mathbf{w}^N_\mu = \mathbf{f}^N_\mu$ of size $N$, where $\mathbf{w}^N_\mu := (w_i(\mu))_{i=1,\ldots,N}\in\R^N$. The reduced basis approximation is then determined as 
\begin{equation*}
	u^N_\mu 
	:= B^*_{\mu} w^N_\mu
	= \sum_{i=1}^N w_i(\mu)\, B^*_{\mu} v^N_i
	= \sum_{i=1}^N \sum_{q=1}^{Q_b} w_i(\mu)\, \theta^q_b(\mu)\, b_{q,i}.
\end{equation*}

\subsection{Basis generation} \label{Subsec:Basis_generation}

While in the standard RB method a reduced trial space is generated from snapshots of the parametrized problem, the reduced discretization of our method is based upon one common reduced test space, while the reduced trial spaces are parameter-dependent. 
However, although we have to find a good basis of the reduced test space $Y^N \subset \cYd$, we still want to build the reduced model from snapshots of the problem. To that end, we again use the formulation \eqref{eq:discrete_param_reform}: Given $\tilde\mu \in \cP$, let $w_{\tilde\mu}^\delta \in \cY^\delta_{\tilde\mu}$ be the solution of \eqref{eq:discrete_param_reform}, such that $u_{\tilde\mu}^\delta := B^*_{\tilde\mu} w_{\tilde\mu}^\delta \in \cX^\delta_{\tilde\mu}$ is the solution of \eqref{eq:discrete_param}.
If $w_{\tilde{\mu}}^\delta \in Y^N$, then we have $u_{\tilde\mu}^\delta \in X^N_{\tilde \mu} = B^*_{\tilde\mu} Y^N$, such that $u^N_{\tilde{\mu}} = u^\delta_{\tilde{\mu}}$ holds for the solution of \eqref{eq:redform}.
Note, however, that due to the parameter dependence of the trial spaces $u_{\tilde\mu}^\delta$ is only included in $X^N_{\tilde\mu}$, but in general $u_{\tilde\mu}^\delta \notin X^N_{\mu}$ for $\mu \neq \tilde{\mu}$ (instead, $B^*_{\mu} w_{\tilde\mu}^\delta \in X^N_\mu$).
Building the reduced test space $Y^N$ from ``snapshots'' of \eqref{eq:discrete_param_reform} is thus analogous to the standard RB strategy to build the reduced trial space from snapshots of the problem of interest: Although a single trial space $X^N_\mu$ is not solely spanned of snapshots, the model error $\|u^N_\mu - u^\delta_\mu\|_{L_2(\Omega)}$ is zero for all parameter values $\mu$ whose \eqref{eq:discrete_param_reform}-snapshot is included in $Y^N$.

\begin{algorithm}[tb]
    \caption{Strong greedy method\label{alg:strong greedy}}
    \begin{algorithmic}[1] 
	\Statex \textbf{input:} train sample $\Xi \subset \mathcal{P}$, tolerance $\varepsilon$ 
	\Statex \textbf{output:} set of chosen parameters $S_{N}$, reduced test space $Y^{N}$ 
	\State \textbf{Initialize} $S_{0} \leftarrow \emptyset$, $Y^{0} \leftarrow \{0\}$
	\ForAll{$\mu \in \Xi$} 
		\State Compute $w^{\delta}_{\mu}$ and $u^{\delta}_{\mu} = B_{\mu}^{*}w^{\delta}_{\mu}$
	\EndFor
	\While{$true$}
		\If{$ \max_{\mu \in \Xi} \| u^{\delta}_{\mu} - u^{N}_{\mu}\|_{L_{2}(\Omega)} \leq \varepsilon$} 
			\State \Return
		\EndIf
		\State $\mu^{*}\leftarrow \arg \max_{\mu \in \Xi} \| u^{\delta}_{\mu} - u^{N}_{\mu}\|_{L_{2}(\Omega)}$\label{alg:mu max}
		\State $S_{N+1} \leftarrow S_{N} \cup \{\mu^{*}\}$
		\State $Y^{N+1} \leftarrow \spanlin \{w^{\delta}_{\mu}, \mu \in S_{N+1}\}$\label{alg:RB test space}
		\State $ N \leftarrow N+1$
	\EndWhile
    \end{algorithmic}
\end{algorithm}

Algorithm \ref{alg:strong greedy} describes an analogue of the standard RB strong greedy algorithm for our setting: Iteratively, we first evaluate the model errors of reduced solutions for all parameters $\mu$ in a train sample $\Xi \subset \cP$. Then, we extend $Y^N$ by the \eqref{eq:discrete_param_reform}-snapshot $w^\delta_{\mu^*} \in \bar\cY^\delta$ corresponding to the worst-approximated parameter $\mu^*$. This automatically extends $X^N_{\mu^*}$ by the \eqref{eq:discrete_param}-snapshot $u^\delta_{\mu^*} \in \cX^\delta_{\mu^*}$, such that from then on the model error for $\mu^*$ is zero.

Of course, this algorithm is computationally expensive, since we have to compute $u^\delta_\mu$ for all $\mu \in \Xi$, which may not be feasible for very complex problems and a finely resolved $\Xi \subset \cP$. It is hence desirable to use some kind of surrogate -- ideally a reliable and efficient error estimator --  instead of the true model error in the greedy algorithm. However, as will be seen in the next subsection, the standard error estimator is not offline-online decomposable in our setting -- a problem already encountered in \cite{DaPlWe14}. Therefore, we have to use error indicators instead when using the full model error is computationally not feasible. We note that until now we are not able to prove convergence of the greedy algorithm due to the parameter-dependent trial spaces. 

Alternatively, to obtain a computational more feasible offline stage one might let the strong greedy run on a small test set with relatively high tolerance and use a hierarchical a posteriori error estimator on the large(r) training set, which was proposed in a slightly different context in \cite{SmeOhl2017}.
Another idea might be to keep a second test training set during the greedy algorithm. In order to estimate the dual norm of the residual more cheaply one could then compute Riesz representations on the span of test training snapshots instead of the full discrete space.

\subsection{Error analysis for the reduced basis approximation} \label{Subsec:RB_Error_Analysis}
In the online stage, for a given (new) parameter $\mu \in \cP$ we are interested in efficiently estimating the model error $\|u^\delta_\mu - u^N_\mu\|_{L_2(\Omega)}$ to assess the quality of the reduced solution. As already mentioned above, due to the choice of the reduced spaces, the reduced inf-sup and continuity constants are unity. This means that the error, the residual, and the error of best approximation coincide also in the reduced setting (cf.\ \eqref{eq:opt}). To be more precise, defining for some $v\in L_2(\Omega)$ the discrete residual $r_{\mu}^\delta(v) \in (\cY^\delta_\mu)'$ as
\begin{equation*}
	\langle r^\delta_{\mu}(v), w^\delta\rangle_{(\mathcal{Y}^\delta_\mu)'\times\cY^\delta_\mu}  
		:= f(w^\delta) - (v, B^*_\mu w^\delta)_{L_2(\Omega)}, 
		\quad w^\delta \in \cY^\delta_\mu,
\end{equation*}
we have 
\begin{equation*}
	\|u^\delta_\mu - u^N_\mu\|_{L_2(\Omega)} 
	= \|r^\delta_{\mu}(u^N_\mu)\|_{(\cY^\delta_\mu)'} 
	= \inf_{v^N \in X^N_\mu} \|u^\delta_\mu- v^N \|_{L_2(\Omega)}.
\end{equation*}
In principle, $r_{\mu}^\delta(v) \in (\cY^\delta_\mu)'$ can  be computed. However, due to the special choice of the parameter-dependent norm of $\cY^\delta_\mu$, i.e., $\|w\|_{\cY^\delta_\mu} = \|B^*_\mu w\|_{L_2(\Omega)}$, the computation of the dual norm involves applying the inverse operator $(B^*_\mu)^{-1}$ and is thus as computationally expensive as solving the discrete problem \eqref{eq:discrete_param}. Therefore, the computation of $\|r^\delta_{\mu}(u^N_\mu)\|_{(\cY^\delta_\mu)'}$ is not offline-online decomposable, so that the residual cannot be computed in an online-efficient manner. 

As an alternative for the error estimation mainly in the online stage, we consider an online-efficient, but non-rigorous \emph{hierarchical error estimator} similar to the one proposed in \cite{BMNP04}.
Let $Y^N \subset Y^M \subset Y^\delta$ be nested reduced spaces with dimensions $N$ and $M$, $N<M$ and denote for some $\mu \in \cP$ by $u^N(\mu) \in X^N_\mu := B^*_\mu Y^N$, $u^M(\mu) \in X^M_\mu := B^*_\mu Y^M$ the corresponding solutions of \eqref{eq:redform}. 
Then, we can rewrite the model error of $u^N$ as 
\begin{equation*}
\|u^N- u^\delta\|_{L_2(\Omega)} = \| u^N - u^M + u^M - u^\delta \|_{L_2(\Omega)}
\leq \| u^N - u^M \|_{L_2(\Omega)} + \| u^M - u^\delta \|_{L_2(\Omega)}.
\end{equation*}
Assuming that $Y^M$ is large enough such that $\| u^M - u^\delta \|_{L_2(\Omega)} < \varepsilon \ll 1$, we can approximate the model error of $u^N$ by
\[
\|u^N- u^\delta\|_{L_2(\Omega)} \leq \| u^N - u^M \|_{L_2(\Omega)} + \varepsilon \approx  \| u^N - u^M \|_{L_2(\Omega)},
\]
which can be computed efficiently also in the online stage. In practice, $Y^N$ and $Y^M$ can be generated by the strong greedy algorithm with different tolerances $\varepsilon_N$ and $\varepsilon_M \ll \varepsilon_N$. Of course, this approximation to the model error is in general not reliable, since it depends on the quality of $Y^M$.  
Reliable and rigorous variants of such an error estimator can be derived based on an appropriate saturation assumption, see \cite{HORU2018}. 
There, also a strategy for the use of hierarchical estimators in terms of Hermite spaces $Y^M$ for the construction of a reduced model in the offline phase have been discussed. We do not go into details here.
Numerical investigations of the quality of the error estimator will be given in \S \ref{subsec:num_rb}.

\section{Computational realization}\label{Sec:5}
	
In this section, we specify the implementation of the solution procedure developed in Section \ref{Sec:3}.  This is also used for the methods for parameter-dependent problems developed in Section \ref{Sec:RBM}. In fact, due to our assumption of affine dependence in the parameter \eqref{eq:affine decomposition}, the computational realization in the parametric setting is very similar to the standard setting and can be done following the offline-online decomposition described at the end of \S \ref{SubSec:4.3}, which is why we do not address it in this section.

To solve the discrete problem \eqref{eq:disc_form} we use the equivalent formulation \eqref{eq:discrete_reformulated}, i.e., we first find $w^\delta \in \cYd$ such that $(B^* w^\delta, B^* v^\delta)_{L_2(\Omega)} = f(v^\delta)$ for all $v^\delta \in \cYd$, and then set $u^\delta := B^* w^\delta \in \cXd$.
The solution procedure thus consists of, first, assembling and solving the problem for $w^\delta$ in $\cYd$, and second, computing $u^\delta$.
The implementation is especially dependent on the exact form of the adjoint operator $B^*$. First, we address the case of constant data, which is easier to implement and slightly more computationally efficient than the general case which we discuss subsequently.

\subsection{Implementation for constant data} \label{Sec:5_const_data} 

We first consider constant data functions in the adjoint operator, which has thus the form
$
B^*w := - \vec{b} \cdot \nabla w + c w 
$
for $0 \neq \vec{b} \in \R^n, c \in \R$. 
We have already seen in Example \ref{Ex_Illustration_Trial} that in the one-dimensional case, choosing a standard linear continuous FE space for the test space $\cYd$ yields a trial space $\cXd$ with piecewise linear and discontinuous functions. This can be generalized to conforming FE test spaces with arbitrary dimension, grid, and polynomial order: If $v^\delta \in \cYd$ is globally continuous and polynomial on each grid cell, all terms of $B^* v^\delta$ are, due to the constant data functions, still polynomials of the same or lower order on the cells, while the gradient terms yield discontinuities on the cell boundaries.  Denoting thus by $\cYd \subset \cY$ a conforming FE space on a partition $\mathcal{T}_\delta = \{K_i\}_{i=1}^{n_{\mathcal{T}_\delta}}$ of $\Omega = \bigcup_{i=1}^{n_{\mathcal{T}_\delta}} K_i$ with polynomial order $p$, and by $\bcXd \subset L_2(\Omega)$ the corresponding discontinuous FE space, i.e.,
\begin{align}
\cYd &:= \{ v \in C^0(\Omega): v|_K \in \mathbb{P}^p(K)\, \forall K \in \mathcal{T}_\delta, v|_{\Gamma_+}=0 \} \subset \cY, \label{eq:def_cYd_FE} \\
\bcXd &:= \{ u \in L_2(\Omega): u|_K \in \mathbb{P}^p(K)\, \forall K \in \mathcal{T}_\delta \} \subset L_2(\Omega), \label{eq:def_bcXd_FE}
\end{align}
we have $\cXd = B^* \cYd \subset \bcXd$ and can determine the solution $u^\delta \in \cXd$ in terms of the standard nodal basis of $\bcXd$.

Let $\mathbf{B^*} \in \R^{\bar{n_x} \times n_y}$ be the matrix representation of 
$B^* : \cYd \to \bcXd$ in the nodal bases $(\phi_1, \dots, \phi_{n_y})$ of $\cYd$ and $(\psi_1, \dots, \psi_{\bar n_x})$ of $\bcXd$, meaning that the $i$-th column of $\mathbf{B^*}$ contains the coefficients of $B^* \phi_i$ in the basis $(\psi_1, \dots, \psi_{\bar n_x})$, i.e., $B^* \phi_i = \sum_{j=1}^{\bar n_x} [\mathbf{B^*}]_{j,i} \psi_j$. Due to the form of the operator and the chosen spaces, the matrix $\mathbf{B^*}$ can be computed rather easily, see the example in \S \ref{Subsubsec:Assemble_matrices_rectangular}.
Then, the coefficient vector $\mathbf{u} = (u_1,\dots, u_{\bar n_x})^T$ of $u^\delta = \sum_{i=1}^{\bar n_x} u_i \psi_i \in \bcXd$ can simply be computed from the coefficient vector $\mathbf{w} = (w_1,\dots,w_{n_y})$ of $w^\delta = \sum_{i=1}^{n_y} w_i \phi_i \in \cYd$ by $\mathbf{u} = \mathbf{B^*} \mathbf{w}$.

To solve \eqref{eq:discrete_reformulated}, we have to assemble the matrix corresponding to the bilinear form $a: \cYd \times \cYd, a(w^\delta, v^\delta) = (B^* w^\delta, B^* v^\delta)_{L_2(\Omega)} = (w^{\delta}, v^\delta)_{\cY}$, i.e., the $\cY$-inner product matrix of $\cYd$. 
One possibility for the assembly is to use the matrix $\mathbf{B^*}$: Denoting by $\mathbf{M_{\bcXd}} \in \R^{\bar n_x \times \bar n_x}$ the $L_2$-mass matrix of $\bcXd$, i.e., $[\mathbf{M_{\bcXd}}]_{i,j} = (\psi_i, \psi_j)_{L_2(\Omega)}$, we see that 
for $\mathbf{Y} := (\mathbf{B^*})^T \mathbf{M_{\bcXd}} \mathbf{B^*} \in \R^{n_y \times n_y}$ it holds $[\mathbf{Y}]_{i,j} = (B^*\phi_i, B^*\phi_j)_{L_2(\Omega)} = (\phi_i,\phi_j)_{\cY}$.

The solution procedure thus consists of the following steps:
\begin{compactenum}
	\item Assemble $\mathbf{B^*}$ and $\mathbf{Y}$
	\item Assemble the load vector $\mathbf{f} \in \R^{n_y}$, $[\mathbf{f}]_i := f(\phi_i), i=1,\dots, n_y$
	\item Solve $\mathbf{Y} \mathbf{w} = \mathbf{f}$
	\item Compute $\mathbf{u} = \mathbf{B^*} \mathbf{w}$
\end{compactenum}

\subsection{Assembling the matrices for spaces on rectangular grids} \label{Subsubsec:Assemble_matrices_rectangular}

As a concrete example on how to assemble the matrices $\mathbf{B^*}$ and $\mathbf{Y}$ we consider $\Omega = (0,1)^n$ and use a rectangular grid.
We start with the one-dimensional case as already seen in Example \ref{Ex_Illustration_Trial}. Let thus $\Omega=(0,1)$ and $b>0$. Moreover, let  $\mathcal{T}^h = \{[(i-1)h,ih)\}_{i=1}^{n_h}$ be the uniform one-dimensional grid with mesh size $h=1/n_h$, fix a polynomial order $p \geq 1$, and define
$\cY^{h,p}_{\oneD}, \bar{\cX}^{h,p}_{\oneD}$ as in \eqref{eq:def_cYd_FE}, \eqref{eq:def_bcXd_FE}.
Let $(\phi_1, \dots, \phi_{n_y})$ and $(\psi_1, \dots, \psi_{\bar n_x})$ be the respective nodal bases of $\cY^{h,p}_{\oneD}$ and $\bar \cX^{h,p}_{\oneD}$.

Moreover, let $\mathbf{I}_{\oneD} \in \R^{\bar n_x \times n_y}$ be the matrix representation of the embedding $\text{Id} : \cY^{h,p}_{\oneD} \to \bar\cX^{h,p}_{\oneD}$ in the respective nodal bases, 
i.e., the $i$-th column of $\mathbf{I}_{\oneD}$ contains the coefficients of $\phi_i \in \cY^{h,p}_{\oneD} \subset \bar\cX^{h,p}_{\oneD}$ in the basis $(\psi_1, \dots, \psi_{\bar n_x})$, 
such that for $\mathbf{u} = \mathbf{I}_{\oneD} \cdot \mathbf{w}$ it holds $\sum_{i=1}^{\bar n_x} u_i \psi_i = \sum_{i=1}^{n_y} w_i \phi_i$. Similarly, let $\mathbf{A}_{\oneD} \in \R^{\bar n_x \times n_y}$ be the matrix representation of the differentiation $\frac{d}{dx}: \cY^{h,p}_{\oneD} \to \bar\cX^{h,p}_{\oneD}, w^h \mapsto (w^h)'$. Additionally, as above, we define $\mathbf{M}_{\oneD} \in \R^{\bar n_x \times \bar n_x}, [\mathbf{M}_{\oneD}]_{i,j} = (\psi_i, \psi_j)_{L_2((0,1))}$ as the $L_2$-mass matrix of $\bar \cX^{h,p}_{\oneD}$.
 
For $p=1$, i.e., linear FE, and a standard choice of the nodal bases the matrices $\mathbf{I}_{\oneD}$, $\mathbf{A}_{\oneD}$, and  $\mathbf{M}_{\oneD}$ read
\begin{equation*}
\mathbf{I}_{\oneD} := 
\begin{pmatrix}
1	&0	&0		&\!\!\cdots \\
0	&1	&0		&		\\
0	&1	&0		&		\\
0	&0	&1		&		\\
\vdots&	&	& \!\!\ddots \\
\end{pmatrix},
\;
\mathbf{A}_{\oneD} := \frac{1}{h} \cdot
\begin{pmatrix}
-1	&1	 &0		&\!\!\cdots \\
-1	&1	 &0		&		\\
0		&\!-1 &1&		\\
0		&\!-1 &1&		\\
\vdots&	&	& \!\!\ddots \\
\end{pmatrix},
\;
\mathbf{M}_{\oneD} = h \cdot
\begin{pmatrix}
\sfrac{1}{3}	& \sfrac{1}{6} & 0 &0 & \!\!\cdots \\
\sfrac{1}{6}	& \sfrac{1}{3} & 0 &0 &  \\
0 & 0 & \sfrac{1}{3}	& \sfrac{1}{6} & \\
0 & 0 & \sfrac{1}{6}	& \sfrac{1}{3} & \\
\vdots&	&&	& \!\!\ddots \\
\end{pmatrix}.
\end{equation*} 
With these three matrices we can then compose the matrices $\mathbf{B}^*_{\oneD}$ and $\mathbf{Y}_{\oneD}$ by
\begin{equation*}
\mathbf{B}^*_{\oneD} := - b \cdot \mathbf{A}_{\oneD} + c \cdot \mathbf{I}_{\oneD}, \qquad
\mathbf{Y}_{\oneD} := (\mathbf{B}^*_{\oneD})^T \mathbf{M}_{\oneD}\mathbf{B}^*_{\oneD}. 
\end{equation*}
 
Next, we consider a rectangular domain of higher dimension, e.g., $\Omega = (0,1)^n, n \geq 2$. We choose in each dimension one-dimensional FE spaces $\cY^i, \bar{\cX}^i$, $i=1,\dots,n$ as in \eqref{eq:def_cYd_FE}, \eqref{eq:def_bcXd_FE} separately, and use the tensor product of these spaces $\cYd := \bigotimes_{i=1}^n \cY^i, \bcXd := \bigotimes_{i=1}^n \bar{\cX}^i$ as FE spaces on the rectangular grid formed by a tensor product of all one-dimensional grids. 
The system matrices can then be assembled from Kronecker products of the one-dimensional matrices corresponding to the spaces $\cY^i, \bar{\cX}^i, i=1,\dots, n$: We first assemble for $i=1,\dots,n$ the matrices $\mathbf{I}_{\oneD}^i$ and $\mathbf{A}_{\oneD}^i$ corresponding to the pair of spaces $\cY^i, \bar{\cX}^i$. Then, the matrix corresponding to the adjoint operator can be assembled by
\begin{equation} \label{eq:matrices_ndim}
\mathbf{B}^* := - \sum_{i=1}^n b_i \mathbf{I}_{\oneD}^1 \otimes \cdots \otimes \mathbf{I}_{\oneD}^{(i-1)} \otimes \mathbf{A}_{\oneD}^i \otimes \mathbf{I}_{\oneD}^{(i+1)}
\otimes  \cdots \otimes \mathbf{I}_{\oneD}^n
+ c \bigotimes_{i=1}^n \mathbf{I}_{\oneD}^i,
\end{equation} 
e.g.\ for $n=2$ we have
\begin{equation*}
\mathbf{B}^*_{\twoD} := - b_1 (\mathbf{A}_{\oneD}^1 \otimes \mathbf{I}_{\oneD}^2)
- b_2 (\mathbf{I}_{\oneD}^1 \otimes \mathbf{A}_{\oneD}^2)
+ c (\mathbf{I}_{\oneD}^1 \otimes \mathbf{I}_{\oneD}^2). 
\end{equation*}
Similarly, the mass matrix $\mathbf{M}$ of $\bcXd$ can be computed from the one-dimensional mass matrices $\mathbf{M}_{\oneD}^i$ of $\bar{\cX}^i, i=1,\dots, n$ by
$\mathbf{M} := \bigotimes_{i=1}^n \mathbf{M}_{\oneD}^i$, such that $\mathbf{Y} := (\mathbf{B}^*)^T \mathbf{M}_{\bcXd} \mathbf{B^*}$ can also be directly assembled using the matrices $\mathbf{I}_{\oneD}^i, \mathbf{A}_{\oneD}^i, \mathbf{M}_{\oneD}^i, i=1, \dots, n$.

\subsection{Implementation for non-constant data}\label{subsect:non_const_data}

If the data functions $\vec{b}$ and $c$ are not constant, we do not automatically get a standard FE space $\bcXd$ in which the solution $u^\delta$ can be described, thus the implementation has to be adapted. 
A way to retain the implementation for constant data functions is to approximate the data by piecewise constants on each grid cell. Then, there holds again $u_\delta \in \bcXd$ and we only have to slightly modify the implementation presented in \S \ref{Sec:5_const_data}:
Every nodal basis function $\psi_i \in \bcXd, i=1,\dots, \bar n_x$, has, due to the discontinuous FE space, a support of only one grid cell. 
Denoting by $c^i$ the value of $c$ on the grid cell of $\psi_i$, we define the diagonal matrix $\mathbf{c} \in \R^{\bar n_x \times \bar n_x}, [\mathbf{c}]_{i,i} := c^i$, and, similarly, the matrices $\mathbf{b}_j\in \R^{\bar n_x \times \bar n_x}$ corresponding to $b_j, j=1,\dots, n$.
We then simply change the scalars $b_j$ and $c$ in \eqref{eq:matrices_ndim} to matrices $\mathbf{b}_j$  and $\mathbf{c}$, $j=1,\dots, n$. 

However, a piecewise constant approximation of the functions $\vec{b} \in C^1(\Omega)^n, c \in C^0(\Omega)$ may not lead to a sufficient accuracy of the solution. 
For general $\vec{b} \in C^1(\Omega)^n, c \in C^0(\Omega)$, we thus first assemble the $\cY$-inner product matrix $\mathbf{Y} \in \R^{n_y \times n_y}$ of $\cYd$ and the load vector $\mathbf{f} \in \R^{n_y}$ corresponding to the right-hand side
as in standard FE implementations for elliptic equations, by using e.g.\ Gauss quadratures for the approximation of the integrals. We can then solve \eqref{eq:discrete_reformulated} as above by $\mathbf{w} := \mathbf{Y}^{-1} \mathbf{f}$, $w^\delta := \sum_{i=1}^{n_y} [\mathbf{w}]_i \phi_i \in \cYd$.
To compute the solution $u_\delta \in \cXd$, we use the fact that we still have $w^\delta \in \bcXd$ and $\frac{\partial w^\delta}{\partial x_i}  \in \bcXd, i=1,\dots, n$, and store the corresponding $\bcXd$-coefficients of $w^\delta$ and its derivatives separately, as well as the data functions.
We can then evaluate $u^{\delta} = B^* w^\delta$ for arbitrary $x \in \Omega$ by evaluating all $w^\delta$-dependent functions and all data functions in $x$ and using the definition of $B^*$ to get $u^\delta(x) = - \sum_{i=1}^n b_i(x) \frac{\partial w^\delta}{\partial x_i}(x) + (c - \nabla \cdot \vec{b})(x) w^\delta(x)$.

\section{Numerical experiments}\label{Sec:6}

In this section, we report on results of our numerical experiments. We consider the parametric and the non-parametric case, starting with the latter one. We are particularly interested in quantitative results concerning the rate of approximation for the discrete case as the discretization parameter $\delta$ (see above) approaches zero, quantitative comparisons of the inf-sup constant with existing methods from the literature and the greedy convergence in the parametric case. We report on time-dependent and time-independent test cases. The source code to reproduce all results is provided in \cite{Bru2018zenodoversion3}.

\subsection{Non-parametric cases} 

\subsubsection{Convergence rates for problems with different smoothness} \label{subsubsec:convergence}

As indicated in \S \ref{SubSec:NewApproach}, we can show the convergence of the proposed approximation for appropriate test spaces $\cYd$, but did not derive theoretical rates of convergence in this paper. Therefore, in this subsection we investigate the rate of convergence in numerical experiments. 
In all test cases we use as test space $\cY^\delta$ a continuous FE space on a uniform hexahedral grid. Since we want to investigate here the best possible convergence rates, we choose test cases where the trial space restrictions due to tensor product spaces described in \S \ref{Subsec:Problem_tensor} do \emph{not} lead to additional errors. These cases will then afterwards be compared to cases where the restrictions indeed \emph{do} lead to additional errors in \S \ref{subsubsec:corner_error_numerics}. 

\begin{figure}[tb] 
\begin{minipage}[t]{0.57\textwidth}
\vspace{0pt}
\begin{center}
\captionsetup{width=.9\linewidth}
	\captionof{table}{\oneD: $L_2$-error and convergence rate as $h\to 0$ for linear and quadratic FE spaces.}\label{Tab:1D}
	{\footnotesize
	\begin{tabular}{|r*2{|r|r|}}\hline
	& \multicolumn{2}{|c||}{Linear FE} &  \multicolumn{2}{|c|}{Quadratic FE}  \\ \cline{2-5}
	$1/h$ & $L_2$-error & rate 		& $L_2$-error	& rate		\\ \hline
	4 	  & 0.03311 	& --- 		& 0.00247		& ---		\\ \hline
	8 	  & 0.01664 	& 0.99274  	& 0.00062		& 1.98932	\\ \hline
	16 	  & 0.00833 	& 0.99817 	& 0.00016		& 1.99729	\\ \hline
	32 	  & 0.00417 	& 0.99954	& 3.896e-05		& 1.99932	\\ \hline
	64 	  & 0.00208  	& 0.99989 	& 9.741e-06		& 1.99983	\\ \hline
	128   & 0.00104 	& 0.99997 	& 2.435e-06		& 1.99996	\\ \hline
	256   & 0.00052		& 0.99999 	& 6.088e-07		& 1.99999	\\ \hline
	\end{tabular}}
\end{center}
\end{minipage}
\hfill
\begin{minipage}[t]{0.4\textwidth}
\vspace{0pt}
	{\footnotesize
	\begin{tikzpicture}\hypersetup{hidelinks}
	\definecolor{c1}{RGB}{228,26,28}
	\definecolor{c2}{RGB}{55,126,184}
	\begin{axis}
	[width=170pt]
	\addplot [color=c1, thick] table [x index={0}, y index={1}, col sep=comma]{oned_8.csv};
	\addplot[color=c2, domain=0:1,samples=100, thick]{exp(-2*x)};
	\end{axis}
	\end{tikzpicture}
	}
	\captionsetup{width=.99\linewidth}
	\captionof{figure}{\oneD: $L_2$-approximation vs.\ exact solution for linear FE space with $h=1/8$.} 
	\label{Fig:1D_Plot}
\end{minipage}
\vspace{-1.5em}
\end{figure}

We start with the one-dimensional problem introduced in Example \ref{Ex_Illustration_Trial} and set $\Omega = (0,1), b(x) \equiv 1, c(x) \equiv 2$ with boundary value $u(0)=1$. We compute approximate solutions for linear FE spaces $\cY^h$ (recall Figure \ref{Fig:1D_Basis_Functions} for the corresponding basis functions and see Figure \ref{Fig:1D_Plot} for an illustration of the solution), as well as quadratic FE spaces. We observe an (optimal) convergence rate of 1 for the linear and 2 for the quadratic case (see Table \ref{Tab:1D}).

Next, we consider $\Omega = (0,1)^2$, and choose $\vec{b} \equiv (\cos 30\degree, \sin 30\degree)^T$, $c\equiv 0, f \equiv 0$, and compare boundary values with different smoothness. In detail, we solve
\begin{equation*}
\vec{b} \cdot \nabla u = 0 \quad \text{in } \Omega , \qquad
u = g^i \quad \text{on } \Gamma_{-} = ( \{0\} \times (0,1) ) \cup ( (0,1) \times \{0\} ), \quad i=1,2,3,
\end{equation*}
for the boundary values 
\begin{align}
g^1 \in C^1(\Gamma_-), \quad
g^1(x,0) &\equiv 1, \quad 
g^1(0,y) = \begin{cases}
31.25 y^3 - 18.75 y^2 + 1, & y \leq 0.4 \label{eq:g^1}\\
0 , & y > 0.4,
\end{cases} \\
g^2 \in C^0(\Gamma_-), \quad
g^2(x,0) &\equiv 1, \quad 
g^2(0,y) = \begin{cases}
1, & y < 0.2 \\
2-5y & 0.2 \leq y < 0.4  \label{eq:g^2}\\
0 , & 0.4 \leq y,
\end{cases} \\
g^3 \in L_2(\Gamma_-), \quad 
g^3(x,0) &\equiv 1, \quad 
g^3(0,y) = \begin{cases}
1, & y < 0.25 \\
0 , & 0.25 \leq y. \label{eq:g^3}
\end{cases}
\end{align}
We use second order FE on a uniform rectangular mesh with $n_h=h^{-1}$ cells in both dimensions, i.e., $\delta = (h,h)$. As already mentioned above, the data is chosen such that for all boundary conditions it holds $u(1,1)=0$ for the exact solution, so that we do not observe problems from the nonphysical restriction of the trial space. We observe a convergence of order about $1.65$ for the differentiable case $g=g^1$, an order of $1$ for the continuous case $g=g^2$, and an order of about $1/3$ for the discontinuous boundary $g=g^3$ (see Table \ref{Tab:2D_g^i}).
\begin{table}[bt]
\begin{center}
\caption{$L_2$-errors and convergence rates for two-dimensional problem with boundary values \eqref{eq:g^1}, \eqref{eq:g^2}, and \eqref{eq:g^3}.}\label{Tab:2D_g^i} 
{\footnotesize
\begin{tabular}{|r*3{|r|r|}}\hline
	& \multicolumn{2}{|c||}{$g=g^1 \in C^1(\Gamma_-)$} &  \multicolumn{2}{|c||}{$g=g^2 \in C^0(\Gamma_-)$}  & \multicolumn{2}{|c|}{$g=g^3 \in L_2(\Gamma_-)$}\\ \cline{2-7}
	$1/h$ & $L_2$-error & rate 		& $L_2$-error	& rate		& $L_2$-error	& rate		\\ \hline
	16 	  & 0.00768 	& ---   	& 0.01974		& ---   	& 0.10630		& ---   	\\ \hline
	32 	  & 0.00247 	& 1.63387	& 0.00973		& 1.02096	& 0.08484		& 0.32533	\\ \hline
	64 	  & 0.00079  	& 1.65196 	& 0.00493		& 0.98128	& 0.06764		& 0.32683	\\ \hline
	128   & 0.00025		& 1.65937	& 0.00248		& 0.99302	& 0.05386		& 0.32862   \\ \hline
	256   & 7.872e-05   & 1.66280   & 0.00124		& 0.99476   & 0.04285		& 0.33009	\\ \hline
	512   & 2.483e-05   & 1.66452   & 0.00062		& 0.99636   & 0.03406		& 0.33120	\\ \hline
\end{tabular}}
\end{center}
\vspace{-0.5em}
\end{table}

To assess the effect of a non-constant transport direction on the convergence rate we use $\vec{b}(x,y) = (1-y,x)^T$, which has an $\Omega$-filling flow with $T = \frac{\pi}{2}$, $c\equiv 0$, $f\equiv 0$, and a $C^1$-boundary value $g^4 \in C^1(\Gamma_-)$ as
\begin{equation*}
g^4(x,0) = 0, \quad g^4(0,y) = 
\begin{cases}
256y^4 - 512y^3 + 352y^2 - 96y + 9, & 0.25 \leq x \leq 0.75, \\
0, & \text{else.}
\end{cases}
\end{equation*}
We observe a convergence behavior even slightly better than for the case of constant $\vec{b}$ with a $C^1$-boundary function, see Table \ref{Tab:2D_Circle}; the curved transport is resolved without artifacts, see Figure \ref{Fig:2D_Circle_Plot}.

\begin{figure}[tb]	
\begin{minipage}{0.48\textwidth} \centering
\captionsetup{width=0.8\linewidth}
	\captionof{table}{$L_2$-error and convergence rate for $\vec{b} = (1-y,x)^T $ and $g=g^4$.}\label{Tab:2D_Circle}
	{\footnotesize
		\begin{tabular}{|r|r|r|}\hline
			$1/h$ & $L_2$-error & rate 		\\ \hline
			4 	  & 0.09317 	& --- 		\\ \hline
			8 	  & 0.03329 	& 1.48458  	\\ \hline
			16 	  & 0.01124 	& 1.56702 	\\ \hline
			32 	  & 0.00366 	& 1.61950	\\ \hline
			64 	  & 0.00117  	& 1.64276 	\\ \hline
			128   & 0.00037 	& 1.65386 	\\ \hline
		\end{tabular}}	
	\end{minipage}
	\hfill
	\begin{minipage}{0.45\textwidth}
		\includegraphics[width=\textwidth]{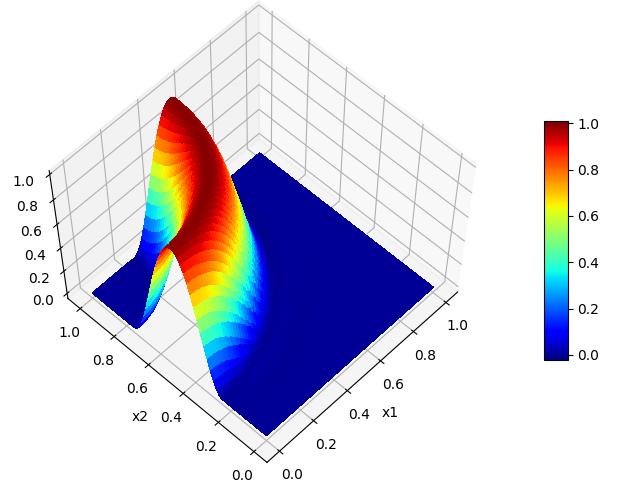}
		\captionsetup{width=.99\linewidth}
		\captionof{figure}{Approximate solution for $\vec{b} = (1-y,x)$, $g=g^4$ and $h=1/32$.} 
		\label{Fig:2D_Circle_Plot}
	\end{minipage}
\vspace{-2.5em}
\end{figure}

\subsubsection{Influence of restrictions due to tensor product spaces} \label{subsubsec:corner_error_numerics}

So far we investigated the convergence of discrete solutions for cases where the nonphysical boundary restrictions described in \S \ref{Subsec:Problem_tensor} do \emph{not} lead to problems. 
Here we want to compare these results to similar test cases where the restriction indeed \emph{is} unphysical, i.e., for the exact solution we have $u \neq 0$ at the relevant outflow boundary part. We again choose $\Omega=(0,1)^2$, $\vec{b} \equiv (\cos 30\degree, \sin 30\degree)^T$, $c \equiv 0,$ and $f \equiv 0$.
We first consider a constant boundary value $\tilde g \equiv 1$, leading to $u \equiv 1$, where the impact of the unphysical restriction can be observed best, since the shifted version $g \equiv 0$ leading to $u \equiv 0$ would of course have no discretization error at all. We compare this to shifted versions of the boundary values considered in \S \ref{subsubsec:convergence}, i.e., $\tilde g^i = g^i - 1, i=1, 2, 3$ for $g^1$, $g^2,$ and $g^3$ defined in \eqref{eq:g^1}-\eqref{eq:g^3}.

\begin{table}[tb]
\caption{$L_2$-errors and convergence rates for two-dimensional problem with different boundary conditions and unphysical restrictions of the trial space.}\label{Tab:2D_corner_restrictions}
\begin{center} 
	{\footnotesize
		\begin{tabular}{|r*4{|r|r|}}\hline
			& \multicolumn{2}{|c||}{$g\equiv 1$} & \multicolumn{2}{|c||}{$g=g^1-1 \in C^1(\Gamma_-)$} &  \multicolumn{2}{|c||}{$g=g^2-1 \in C^0(\Gamma_-)$}  & \multicolumn{2}{|c|}{$g=g^3-1 \in L_2(\Gamma_-)$}\\ \cline{2-9}
			$1/h$ & $L_2$-error & rate 		& $L_2$-error & rate 	& $L_2$-error	& rate		& $L_2$-error	& rate		\\ \hline
			16 	  & 0.01280 	& ---   	& 0.01479 	  & ---     & 0.02627		& ---   	& 0.10618		& ---   	\\ \hline
			32 	  & 0.00676 	& 0.92191	& 0.00691 	  & 1.09798	& 0.01281		& 1.03615	& 0.08515		& 0.31838	\\ \hline
			64 	  & 0.00355 	& 0.92883	& 0.00349  	  & 0.98507 & 0.00616		& 1.05729	& 0.06773		& 0.33028	\\ \hline
			128   & 0.00186 	& 0.93469	& 0.00183	  & 0.92944	& 0.00292		& 1.07500	& 0.05389		& 0.32963	\\ \hline
			256   & 0.00097 	& 0.93973	& 0.00097     & 0.92081 & 0.00149       & 0.97073   & 0.04286       & 0.33058   \\ \hline
			512   & 0.00050     & 0.94411   & 0.00050     & 0.94099 & 0.00081	    & 0.88878   & 0.03406       & 0.33141   \\ \hline
		\end{tabular}}
\end{center}
\vspace{-0.5em} 
	\end{table}

In the constant case $g \equiv 1$ we have a convergence of order $\approx \!1$ (see Table \ref{Tab:2D_corner_restrictions}). Comparing Tables \ref{Tab:2D_g^i} and \ref{Tab:2D_corner_restrictions}, we see that indeed the restriction leads to an additional error that converges with order 1: While the problem for the $C^1$-boundary value $g^1$ converges with an order of about $1.65$, the shifted problem for $\tilde g^1 = g^1-1$ converges only with an order of $\approx\!1$. For the less smooth boundaries $\tilde g^2 \in C^0(\Gamma_-)$ and $\tilde g^3 \in L_2(\Gamma_-)$ we see that the convergence order stays the same, and thus the full error is not dominated by the restriction artifacts. 
All in all, for the present test cases, the restriction due to the tensor product structure limits the convergence rate to 1, but does not deteriorate smaller convergence orders for less smooth problems, such that for these problems the additional error is negligible. Recall that we are primarily interested in such non-smooth solutions in $L_2(\Omega)$.

Next, we investigate the approach proposed in \S \ref{Subsec:Problem_tensor} to use an additional layer for the computational domain. In detail, we extend the data functions onto the larger domain $\Omega(\alpha)$ defined in \eqref{eq:extended_domain}, solve the problem for a discrete solution $u^\delta_{\Omega(\alpha)} \in L_2(\Omega(\alpha))$ on this extended problem, and then define the restriction $u^\delta_{\Omega(\alpha)}|_{\Omega} \in L_2(\Omega)$ as discrete solution to the original problem.

We consider constant boundary values $g \equiv 1$. For each discrete space $\cY^\delta, \delta = (h,h)$, we compare values of $\alpha = mh$, $m=1,\dots, 5$, i.e., we extend the domain by 1 to 5 layers of grid cells of the original size. The $L_2$- and $L_\infty$-errors of these solutions and the respective solutions computed on the original domain $\Omega$ are shown in \ref{Fig:Ext_dom_one_errors}. 
We see that using extended domains for the computation reduces the $L_2$-errors: A first layer of grid cells has the most significant effect, but also larger extensions further reduce the errors. Since the difference is larger for coarser meshes, the $L_2$-rates are slightly lower than for the original solution, which improves however for finer mesh sizes. We obtained similar results for the boundary values $g = g^{1} - 1$. Moreover, the extended domain approach has a positive impact on the $L_{\infty}$-error of the solution and thus on the ``optical quality'': While for the computations on $\Omega$ we automatically have an $L_{\infty}$-error of $1$ for all mesh sizes, the error is reduced to values between about $0.16$ for $\alpha=h$ and $0.05$ for $\alpha=5h$; also the $L_{\infty}$-error on the extended domains seems to be relatively independent of the mesh size (see \ref{Fig:Ext_dom_one_errors}). A comparison of the solution computed on $\Omega$ and $\Omega(h)$ is provided in \ref{Fig:2D-g_const-std-ext}.

We conclude from these experiments that for the current test case the use of an extended domain slightly reduces the $L_2$-error while maintaining comparable convergence rates and considerably reduces the $L_\infty$-error at the boundary.
Hence, at the expense of (moderate) additional computational cost a better approximation of the solution on the outflow boundary can be achieved.

\begin{figure}
\begin{center}
\begin{tikzpicture}[scale=0.85] \hypersetup{hidelinks}
\pgfplotsset{
  log x ticks with fixed point/.style={
      xticklabel={
        \pgfkeys{/pgf/fpu=true}
        \pgfmathparse{exp(\tick)}%
        \pgfmathprintnumber[fixed relative, precision=3]{\pgfmathresult}
        \pgfkeys{/pgf/fpu=false}
      }
  },
  log y ticks with fixed point/.style={
      yticklabel={
        \pgfkeys{/pgf/fpu=true}
        \pgfmathparse{exp(\tick)}%
        \pgfmathprintnumber[fixed relative, precision=3]{\pgfmathresult}
        \pgfkeys{/pgf/fpu=false}
      }
  }
}
		\definecolor{c1}{RGB}{228,26,28}
		\definecolor{c2}{RGB}{55,126,184}
		\definecolor{c3}{RGB}{77,175,74}
		\definecolor{c4}{RGB}{152,78,163}
		\definecolor{c5}{RGB}{255,127,0}
		\definecolor{c6}{RGB}{166,86,40}
		\begin{loglogaxis}
		[width=160pt, xlabel=$h^{-1}$, ylabel=$L_2$-error, legend style={at={(1.03,0.5)}, anchor=west}, xtick={16, 32, 64, 128, 256, 512}, log x ticks with fixed point,]
		\addplot[color=c1, thick, solid, mark=*] table [x index={0}, y index={1}, col sep=comma]{ext_grid_one_layer_comp.csv};
		\addplot[color=c2, thick, solid, mark=square] table [x index={0}, y index={2}, col sep=comma]{ext_grid_one_layer_comp.csv};
		\addplot[color=c3, thick, solid, mark=triangle*] table [x index={0}, y index={3}, col sep=comma]{ext_grid_one_layer_comp.csv};
		\addplot[color=c4, thick, solid, mark=diamond] table [x index={0}, y index={4}, col sep=comma]{ext_grid_one_layer_comp.csv};
		\addplot[color=c5, thick, solid, mark=pentagon*] table [x index={0}, y index={5}, col sep=comma]{ext_grid_one_layer_comp.csv};
		\addplot[color=c6, thick, solid, mark=asterisk] table [x index={0}, y index={6}, col sep=comma]{ext_grid_one_layer_comp.csv};
		\end{loglogaxis}
\end{tikzpicture}
\begin{tikzpicture}[scale=0.85] \hypersetup{hidelinks}
\pgfplotsset{
  log x ticks with fixed point/.style={
      xticklabel={
        \pgfkeys{/pgf/fpu=true}
        \pgfmathparse{exp(\tick)}%
        \pgfmathprintnumber[fixed relative, precision=3]{\pgfmathresult}
        \pgfkeys{/pgf/fpu=false}
      }
  },
  log y ticks with fixed point/.style={
      yticklabel={
        \pgfkeys{/pgf/fpu=true}
        \pgfmathparse{exp(\tick)}%
        \pgfmathprintnumber[fixed relative, precision=3]{\pgfmathresult}
        \pgfkeys{/pgf/fpu=false}
      }
  }
}
		\definecolor{c1}{RGB}{228,26,28}
		\definecolor{c2}{RGB}{55,126,184}
		\definecolor{c3}{RGB}{77,175,74}
		\definecolor{c4}{RGB}{152,78,163}
		\definecolor{c5}{RGB}{255,127,0}
		\definecolor{c6}{RGB}{166,86,40}
		\begin{loglogaxis}
		[width=160pt, xlabel=$h^{-1}$, ylabel=$L_\infty$-error, legend style={at={(1.03,0.5)}, anchor=west}, xtick={16, 32, 64, 128, 256, 512}, log x ticks with fixed point,]
		\addplot[color=c1, thick, solid, mark=*] table [x index={0}, y index={1}, col sep=comma]{ext_grid_one_layer_comp_l_infty.csv};
		\addlegendentry{Standard};
		\addplot[color=c2, thick, solid, mark=square] table [x index={0}, y index={2}, col sep=comma]{ext_grid_one_layer_comp_l_infty.csv};
		\addlegendentry{$\alpha = h$};
		\addplot[color=c3, thick, solid, mark=triangle*] table [x index={0}, y index={3}, col sep=comma]{ext_grid_one_layer_comp_l_infty.csv};
		\addlegendentry{$\alpha = 2h$};
		\addplot[color=c4, thick, solid, mark=diamond] table [x index={0}, y index={4}, col sep=comma]{ext_grid_one_layer_comp_l_infty.csv};
		\addlegendentry{$\alpha = 3h$};
		\addplot[color=c5, thick, solid, mark=pentagon*] table [x index={0}, y index={5}, col sep=comma]{ext_grid_one_layer_comp_l_infty.csv};
		\addlegendentry{$\alpha = 4h$};
		\addplot[color=c6, thick, solid, mark=asterisk] table [x index={0}, y index={6}, col sep=comma]{ext_grid_one_layer_comp_l_infty.csv};
		\addlegendentry{$\alpha = 5h$};
		\end{loglogaxis}
\end{tikzpicture}
\end{center}
\caption{2D, $g \equiv 1$, $L_2$-errors (left) and $L_\infty$-errors (right) for solutions computed on the standard domain $\Omega = \Omega(0)$ and on extended domains $\Omega(\alpha)$, $\alpha = mh$, $m=1,\dots,5$ for different mesh sizes. \label{Fig:Ext_dom_one_errors}}
\vspace{-1.5em}
\end{figure}
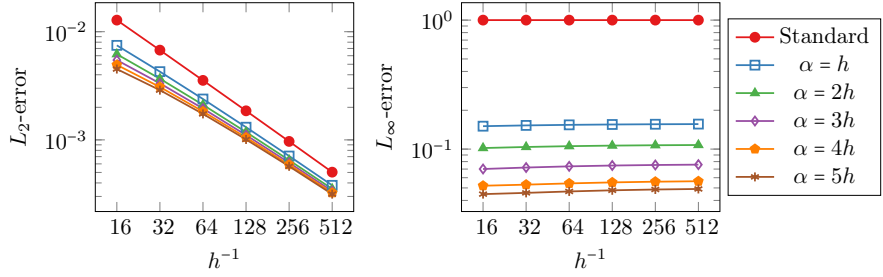 

\begin{figure}[tb]
\centering
	\includegraphics[width=0.9\textwidth]{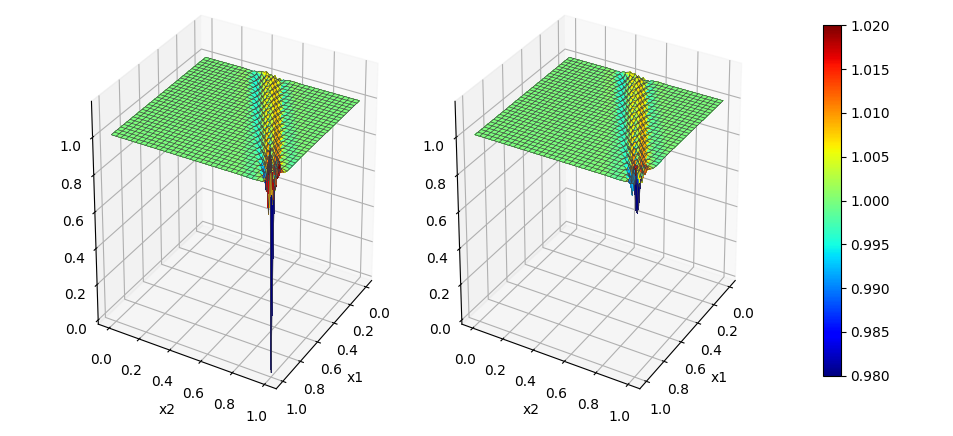}
	\caption{2D, Numerical approximation for $h=1/32$, $g\equiv 1$. Left: Standard domain $\Omega$. Right: $u^\delta|_{\Omega}$ solved on extended domain $\Omega(h)$}\label{Fig:2D-g_const-std-ext}
\vspace{-1.5em}
\end{figure}

\subsubsection{Assessment of post-processing procedure}

We next compare the approximation of discontinuities of a standard solution $u^\delta \in \cXd$ to the post-processed solution $\tilde u^\delta$ described in \S\ref{Subsec:Post-processing}. 
To this end, we again consider the example in \S \ref{subsubsec:convergence} with boundary value $g^3 \in L_2(\Gamma_-)$ that is piecewise constant with a discontinuity.
Note that the choice of a constant advection $\vec{b}$ and no reaction simplifies the post-processing procedure, such that the post-processed solution $\tilde u^\delta$ directly is the $L_2$-orthogonal projection of $u^\delta$ onto the discontinuous first order FE space.  Comparing the errors of $u^\delta$ and $\tilde u^\delta$ (see Tables \ref{Tab:2D_g^i} and \ref{Tab:Post_processed}), we see that the errors for the post-processed solutions are by about 8\% smaller than for the standard solutions, while the order of convergence stays the same. Figure \ref{Fig:Standard_Post-processed} shows that the post-processing removes the severe overshoots of the standard solution at the jump discontinuity. 
We also note that the post-processing is computationally inexpensive, since it is only based upon local multiplications of an element projection matrix for each grid cell. A comparison of the computational costs will be given in \S \ref{subsec:comparison_Dahmen_numerics}.
\begin{figure}[tb]
	\begin{minipage}{0.3\textwidth} \centering
		\captionsetup{width=\linewidth}
		\captionof{table}{$L_2$-error and convergence rate for post-processed solution $\tilde u^\delta$ for boundary $g^3$.}\label{Tab:Post_processed}
		{\footnotesize
		\begin{tabular}{|r|r|r|}\hline
			$1/h$ & $L_2$-error & rate 		\\ \hline
			16 	  & 0.09769 	& ---    	\\ \hline
			32 	  & 0.07765		& 0.33128	\\ \hline
			64 	  & 0.06179  	& 0.32946 	\\ \hline
			128   & 0.04917 	& 0.32965 	\\ \hline
			256   & 0.03911		& 0.33042	\\ \hline
			512   & 0.03108		& 0.33123	\\ \hline
		\end{tabular}}			
	\end{minipage}
	\hfill
	\begin{minipage}{0.69\textwidth}
		\includegraphics[trim= 50 0 90 0, clip, width=\textwidth]{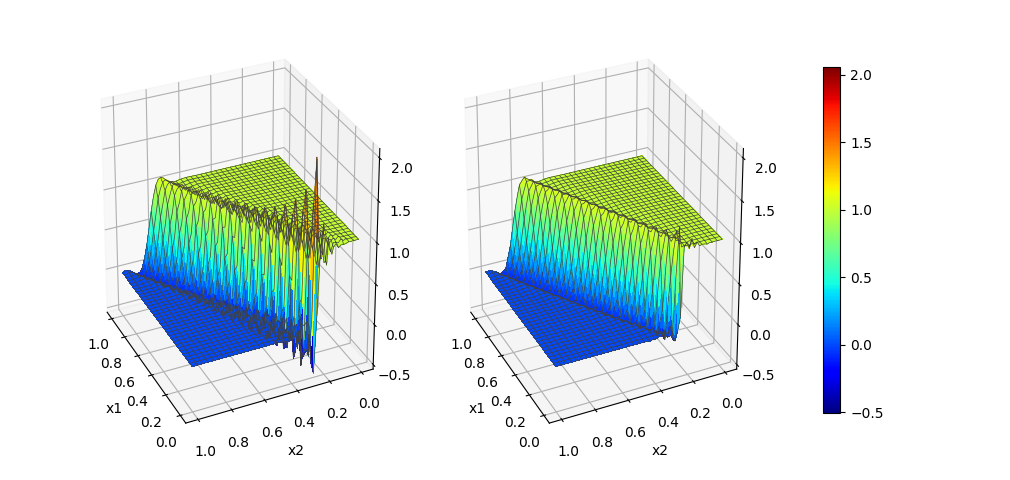}
		\captionsetup{width=.99\linewidth}
		\captionof{figure}{Standard solution $u^\delta$ (left) and post-processed solution $\tilde u^\delta$ (right) for boundary $g^3$ and $h=1/32$.} 
		\label{Fig:Standard_Post-processed}
	\end{minipage}
\vspace{-1.5em}
\end{figure}

\subsubsection{Comparison to approach proposed in \cite{DHSW2012}} \label{subsec:comparison_Dahmen_numerics}

Next, we compare the results of our method, which we call \emph{Optimal Trial} method, 
with the related approach in \cite{DHSW2012}, which we call \emph{Optimal Test} method. 
We use the same test case as in \cite{DHSW2012}, i.e., we set $\Omega = (0,1)^2$, $\vec{b} \equiv (\cos 22.5\degree, \sin 22.5\degree)^T$, $c\equiv 0$, and $f \equiv 0$. For the boundary condition we again have the discontinuous boundary value $g = g^3$ defined in \eqref{eq:g^3}.

While our \emph{Optimal Trial} approach consists of choosing a test space $\cYd \subset \cY$ which automatically determines the trial space $\cXd = B^* \cYd \subset L_2(\Omega)$ and the corresponding linear system in $\cYd$, 
for the \emph{Optimal Test} method in \cite{DHSW2012} one chooses a trial space $\widehat{\mathcal{X}}^\delta\subset L_2(\Omega)$ and a larger test search space $\mathcal{Z}^\delta \subset \cY$, i.e., $\cYd\subset\mathcal{Z}^\delta$. The optimally stable problem in $\widehat{\mathcal{X}}^\delta \times (B^*)^{-1} \widehat{\mathcal{X}}^\delta$ is then substituted by the problem in $\widehat{\mathcal{X}}^\delta \times P_{\mathcal{Z}^\delta}((B^*)^{-1} \widehat{\mathcal{X}}^\delta)$, which in turn is solved \emph{approximately} by an Uzawa algorithm. 
Within this algorithm, one iteratively solves problems in the test space $\mathcal{Z}^\delta$, which are in fact based upon the same bilinear form as for \eqref{eq:discrete_reformulated} to be solved in $\cYd$ in the \emph{Optimal Trial} method. We therefore choose the spaces such that $\cYd = \mathcal{Z}^\delta$, which means that the same matrix has to be assembled for both methods.  
More precisely, we choose for the \emph{Optimal Trial} method the same spaces as in the experiments above, i.e., $\cYd$ is the space of continuous FE of second order on a rectangular grid with mesh size $\delta = (h,h)$. Fitting to that, we choose -- as proposed in \cite{DHSW2012} -- for $\widehat{\mathcal{X}}^\delta$ the space of discontinuous bilinear FE on a rectangular grid with mesh size $(2h, 2h)$, and $\mathcal{Z}^\delta = \cYd$, such that here the grid for the test search space results from one uniform refinement of the grid of the trial space.

\begin{figure}[tb]
\begin{center}
	\footnotesize{
\begin{tikzpicture} \hypersetup{hidelinks}
	\definecolor{c1}{RGB}{228,26,28}
	\definecolor{c2}{RGB}{55,126,184}
	\definecolor{c3}{RGB}{77,175,74}
	\definecolor{c4}{RGB}{152,78,163}
	\definecolor{c5}{RGB}{255,127,0}
	\begin{loglogaxis}[width=185pt, legend style={at={(1.03,0.5)}, anchor=west}, xlabel=CPU time, ylabel=$L_2$-error]
	\addplot [color=c1, mark=*, mark size=1.5pt, thick] table [x index={1}, y index={2}, col sep=comma]{cpu_factorized_umfpack.csv};
	\addlegendentry{\emph{Opt.\ Test} 1;}
	\addplot [color=c2, mark=square, mark size=1.5pt, thick] table [x index={3}, y index={4}, col sep=comma]{cpu_factorized_umfpack.csv};
	\addlegendentry{\emph{Opt.\ Test} 5;}
	\addplot [color=c3, mark=triangle*, thick] table [x index={5}, y index={6}, col sep=comma]{cpu_factorized_umfpack.csv};
	\addlegendentry{\emph{Opt.\ Trial};}
	\addplot [color=c4, mark=diamond, thick] table [x index={7}, y index={8}, col sep=comma]{cpu_factorized_umfpack.csv};
	\addlegendentry{\emph{Opt.\ Trial} post-proc.}
\end{loglogaxis}
\end{tikzpicture}
}
\end{center}
\caption{$L_2$-errors versus CPU-times for the \emph{Optimal Test} method with 1 iteration (\emph{Opt.\ Test} 1) and 5 iterations (\emph{Opt.\ Test} 5) of the Uzawa algorithm, and for the \emph{Optimal Trial} method in standard (\emph{Opt.\ Trial}) and post-processed (\emph{Opt.\ Trial} post-proc.) form.}
\label{Figure:CPU_Errors}
\vspace{-2em} 
\end{figure}
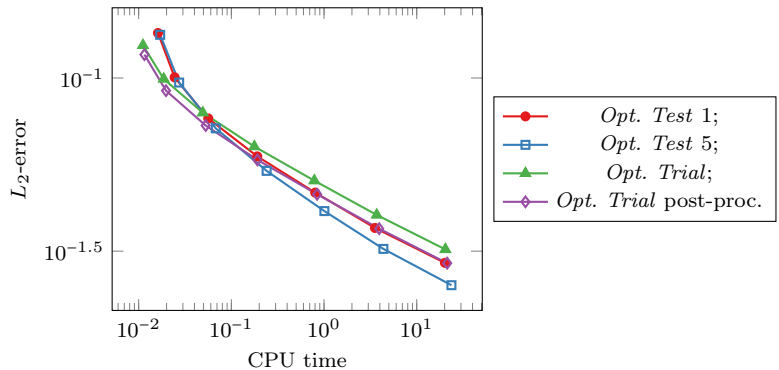

We first compare the relation of $L_2$-errors and CPU times for both methods. For the solution of the linear systems, we always use sparse LU factorization and subsequent forward and back substitution implemented in UMFPACK.
Figure \ref{Figure:CPU_Errors} shows the respective CPU-error plots for the \emph{Optimal Test} method using 1 iteration and 5 iterations of the Uzawa algorithm (as proposed in \cite{DHSW2012}) and for the standard solution of the \emph{Optimal Trial} method as well as the post-processed solution described in \S \ref{Subsec:Post-processing}.
We observe similar decay rates of the errors w.r.t.\ the CPU times for both methods.
For the chosen linear solver, the \emph{Optimal Test} methods with 5 iterations performs best, which is mainly due to the fact that assembly of the matrices and LU factorization dominate the computational costs. Therefore, the costs for 5 Uzawa iterations are only slightly higher than for e.g.\ only 1 Uzawa iteration, while the errors are reduced significantly.     
If we use iterative methods, e.g.\ the CG method, instead, the results depend on the used preconditioner: If the computation of the preconditioner dominates, the results are similar to the results using LU decomposition. In contrast, if the iterative solver takes as much or more time than the preconditioner, then the \emph{Optimal Test} solutions using 5 Uzawa iterations would take considerably more time compared to the other solutions and we speculate that the post-processed \emph{Optimal Trial} solution might perform fairly equally to the \emph{Optimal Test} solutions. However, a comparison of different preconditioners is out of the scope of this paper.
 
Finally, we compare the inf-sup constants of both methods. While for the \emph{Optimal Trial} method we automatically have an inf-sup constant of 1, this is not the case for the \emph{Optimal Test} method. Since here not the truly optimal test space $(B^*)^{-1} \widehat{\mathcal{X}}^\delta$, but the projection onto the test search space $P_{\mathcal{Z}^\delta}((B^*)^{-1} \widehat{\mathcal{X}}^\delta)$ is used for the discrete test space, the inf-sup constant for the discrete problem as well as for the corresponding saddle-point problem on which the Uzawa iteration is based is suboptimal. Table \ref{Tab:Inf_Sup_2D} and Table \ref{Tab:Inf_Sup_3D} show the inf-sup constants for the considered two-dimensional problem, i.e., $\Omega = (0,1)^2$, $\vec{b} = (\cos 22.5\degree, \sin 22.5\degree)^T$, $c\equiv 0$ and the corresponding time-dependent problem, i.e., a three-dimensional problem with 
$\Omega = (0,1)^3$ and $\vec{b} = (1, \cos 22.5\degree, \sin 22.5\degree)^T$, respectively. We clearly see that the inf-sup constants decrease with smaller mesh sizes, in both cases they decay roughly with an order of $h^{1/3}$.

\begin{table}[tb]
\centering
	\begin{minipage}{0.4\linewidth}
		\captionsetup{width=\linewidth}
		\captionof{table}{Inf-sup constants for the \emph{Optimal Test} method and the 2D problem} 
		\label{Tab:Inf_Sup_2D}
		\centering
		{\footnotesize
		\begin{tabular}{|r|r|}
			\hline
			1/(2h)  & Inf-sup \\\hline
			4   &  0.74521 \\\hline
			8   &  0.66426 \\\hline
			16  &  0.55840 \\\hline
			32  &  0.45422 \\\hline
			64  &  0.36029 \\\hline
			128 &  0.28273 \\\hline
			256 &  0.21901 \\\hline
		\end{tabular}}
	\end{minipage}
\hspace{0.6cm}
	\begin{minipage}{0.4\linewidth}
		\captionsetup{width=\linewidth}
		\captionof{table}{Inf-sup constants for the \emph{Optimal Test} method and the 3D problem} 
		\label{Tab:Inf_Sup_3D}
		\centering
		{\footnotesize
		\begin{tabular}{|r|r|}
			\hline
			1/(2h)  & Inf-sup \\\hline
			4   &  0.64800 \\\hline
			8   & 0.60160 \\\hline
			16  & 0.48294 \\\hline
			32  & 0.38015 \\\hline
		\end{tabular}}
	\end{minipage}
\vspace{-1.5em}
\end{table}

\subsection{Parametric cases: The reduced basis method}
\label{subsec:num_rb}

\hspace{-0.5mm}To examine our method in the parametric setting, we consider three different test cases. For all cases, we choose $\Omega = (0,1)^2$ and a parametrized constant transport direction $\vec{b}_\mu \in \R^2, \mu \in \cP$, such that $\Gamma_- = (\{0\} \times (0,1) ) \cup ( (0,1) \times \{0\} )$ for all $\mu \in \cP$ as well as parameter-independent reaction, source and boundary data, see Table \ref{Tab:ParaCase}.
Again, we want to solve for all $\mu \in \cP$
\begin{equation*}
\vec{b}_\mu \cdot \nabla u +cu = f \quad \text{in }\Omega,
\qquad
u=g \quad \text{on }\Gamma_{-}.
\end{equation*}

\begin{table}[tb]
\caption{\label{Tab:ParaCase}Data for parametric test cases.}
\begin{center}
\begin{tabular}{|c|l|l|l|}\hline
	 & Test Case 1 (see \cite{OR16}) &  Test Case 2 (cf.\ \cite{DaPlWe14})  & Test Case 3 (cf.\ \cite{DaPlWe14})\\ \hline\hline
	$\vec{b}_\mu$ & $(\mu, 1)^T$ 	& $(\cos\mu, \sin\mu)^T$ 	& $(\cos\mu, \sin\mu)^T$ 	\\ \hline
	$\cP$ 		& $[0.01, 1]$ 	& $[0.2, \frac{\pi}{2}-0.2]$	& $[0.2, \frac{\pi}{2}-0.2]$ \\ \hline
	$c$			& $ \equiv 0$	& $\equiv 1$			& $\equiv 1$  \\ \hline
	$f$			& $\equiv 0$	& $\equiv 1$			& $\begin{cases}0.5,&x<y\\1,&x\geq y\end{cases}$ \\ \hline
	$g$			& $\begin{cases}1,&x=0\\0,&y=0\end{cases}$ 
				& $\equiv 0$
				& $\begin{cases}1-y,&x\leq 0.5\\0, & x \geq 0.5\end{cases}$ \\ \hline
\end{tabular}
\end{center}
\vspace{-0.5em}
\end{table}

For all test cases, we choose a training set of 500 equidistant parameter values distributed over $\cP$ and set $\varepsilon=10^{-4}$. We then generate reduced models with Algorithm \ref{alg:strong greedy} for different mesh sizes. The maximum model errors $\|u^N(\mu)-u^\delta(\mu)\|_{L_2(\Omega)}$ on an additional test set of 500 uniformly distributed random parameter values are shown in Figure \ref{Fig:RB_results}.
\begin{figure}[tb]
	\pgfplotsset{
		compat=1.11,
		legend image code/.code={
			\draw[mark repeat=2,mark phase=2]
			plot coordinates {
				(0cm,0cm)
				(0.15cm,0cm)        
				(0.3cm,0cm)         
			};%
		}
	}
\begin{center}
	\footnotesize{
		\begin{tikzpicture} \hypersetup{hidelinks}
		\definecolor{c1}{RGB}{228,26,28}
		\definecolor{c2}{RGB}{55,126,184}
		\definecolor{c3}{RGB}{77,175,74}
		\definecolor{c4}{RGB}{152,78,163}
		\definecolor{c5}{RGB}{255,127,0}
		\definecolor{c6}{RGB}{166,86,40}
		\begin{loglogaxis}
		[width=135pt, title=Test Case 1, xlabel=$N$, ylabel=$L_2$-model error]
		\addplot[color=c1, thick, densely dotted] table [x index={0}, y index={1}, col sep=comma]{RB1_greedy_test_equid_500_test_random_500.csv};
		\addplot[color=c2, thick, densely dashdotted] table [x index={0}, y index={2}, col sep=comma]{RB1_greedy_test_equid_500_test_random_500.csv};
		\addplot[color=c3, thick, solid] table [x index={0}, y index={3}, col sep=comma]{RB1_greedy_test_equid_500_test_random_500.csv};
		\addplot[color=c4, thick, densely dotted] table [x index={0}, y index={4}, col sep=comma]{RB1_greedy_test_equid_500_test_random_500.csv};
		\addplot[color=c5, thick, densely dashdotted] table [x index={0}, y index={5}, col sep=comma]{RB1_greedy_test_equid_500_test_random_500.csv};
		\addplot[color=c6, thick, solid] table [x index={0}, y index={6}, col sep=comma]{RB1_greedy_test_equid_500_test_random_500.csv};
		\addplot[red, thick, domain=1:50, samples=200,]{x^(-1/2)};
		\label{ord_05}
		\node [draw,fill=white,font=\tiny] at (rel axis cs: 0.28,0.2) {\shortstack[l]{
				\ref{ord_05} $\sim N^{-1/2}$}};
		\end{loglogaxis}
		\end{tikzpicture}
		\begin{tikzpicture} \hypersetup{hidelinks}
		\definecolor{c1}{RGB}{228,26,28}
		\definecolor{c2}{RGB}{55,126,184}
		\definecolor{c3}{RGB}{77,175,74}
		\definecolor{c4}{RGB}{152,78,163}
		\definecolor{c5}{RGB}{255,127,0}
		\definecolor{c6}{RGB}{166,86,40}
		\begin{loglogaxis}
		[width=135pt, title=Test Case 2, xlabel=$N$]
		\addplot [color=c1, thick, densely dotted] table [x index={0}, y index={1}, col sep=comma]{RB2_greedy_test_equid_500_test_random_500.csv};
		\addplot [color=c2, thick, densely dashdotted] table [x index={0}, y index={2}, col sep=comma]{RB2_greedy_test_equid_500_test_random_500.csv};
		\addplot [color=c3, thick, solid] table [x index={0}, y index={3}, col sep=comma]{RB2_greedy_test_equid_500_test_random_500.csv};
		\addplot [color=c4, thick, densely dotted] table [x index={0}, y index={4}, col sep=comma]{RB2_greedy_test_equid_500_test_random_500.csv};
		\addplot [color=c5, thick, densely dashdotted] table [x index={0}, y index={5}, col sep=comma]{RB2_greedy_test_equid_500_test_random_500.csv};
		\addplot [color=c6, thick, solid] table [x index={0}, y index={6}, col sep=comma]{RB2_greedy_test_equid_500_test_random_500.csv};
		\addplot[red, thick, domain=1:50, samples=200,]{1^(3/2)*x^(-3/2)};
		\node [draw,fill=white,font=\tiny] at (rel axis cs: 0.28,0.2) {\shortstack[l]{
				\ref{ord_05} $\sim N^{-3/2}$}};	
		\end{loglogaxis}
		\end{tikzpicture}
		\begin{tikzpicture} \hypersetup{hidelinks}
		\definecolor{c1}{RGB}{228,26,28}
		\definecolor{c2}{RGB}{55,126,184}
		\definecolor{c3}{RGB}{77,175,74}
		\definecolor{c4}{RGB}{152,78,163}
		\definecolor{c5}{RGB}{255,127,0}
		\definecolor{c6}{RGB}{166,86,40}
		\begin{loglogaxis}
		[width=135pt, legend columns=-1, legend to name=legendname, title={Test Case 3}, xlabel=$N$]
		\addplot [color=c1, densely dotted, thick] table [x index={0}, y index={1}, col sep=comma]{RB3_greedy_test_equid_500_test_random_500.csv};
		\addlegendentry{$h^{-1}=16$;}
		\addplot [color=c2, thick, densely dashdotted] table [x index={0}, y index={2}, col sep=comma]{RB3_greedy_test_equid_500_test_random_500.csv};
		\addlegendentry{$h^{-1}=32$;}
		\addplot [color=c3, thick, solid] table [x index={0}, y index={3}, col sep=comma]{RB3_greedy_test_equid_500_test_random_500.csv};
		\addlegendentry{$h^{-1}=64$;}
		\addplot [color=c4, thick, densely dotted] table [x index={0}, y index={4}, col sep=comma]{RB3_greedy_test_equid_500_test_random_500.csv};
		\addlegendentry{$h^{-1}=128$;}
		\addplot [color=c5, thick, densely dashdotted] table [x index={0}, y index={5}, col sep=comma]{RB3_greedy_test_equid_500_test_random_500.csv};
		\addlegendentry{$h^{-1}=256$;}
		\addplot [color=c6, thick] table [x index={0}, y index={6}, col sep=comma]{RB3_greedy_test_equid_500_test_random_500.csv};
		\addlegendentry{$h^{-1}=512$}
		\addplot[red, thick, domain=2:80, samples=200,]{2*x^(-1)};
		\node [draw,fill=white,font=\tiny] at (rel axis cs: 0.28,0.2) {\shortstack[l]{
				\ref{ord_05} $\sim N^{-1}$}};
		\end{loglogaxis}
		\end{tikzpicture}
		\centering
		\begin{tikzpicture} 
		\hypersetup{hidelinks}
		\definecolor{c1}{RGB}{228,26,28}
		\definecolor{c2}{RGB}{55,126,184}
		\definecolor{c3}{RGB}{77,175,74}
		\definecolor{c4}{RGB}{152,78,163}
		\definecolor{c5}{RGB}{255,127,0}
		\definecolor{c6}{RGB}{166,86,40}
		\ref{legendname}
		\end{tikzpicture}
	}
\end{center}
	\caption{Maximum errors of 500 test parameter values for different model orders, mesh sizes, and Test Cases 1, 2, and 3.
	\label{Fig:RB_results}}
\vspace{-1.5em} 
\end{figure}
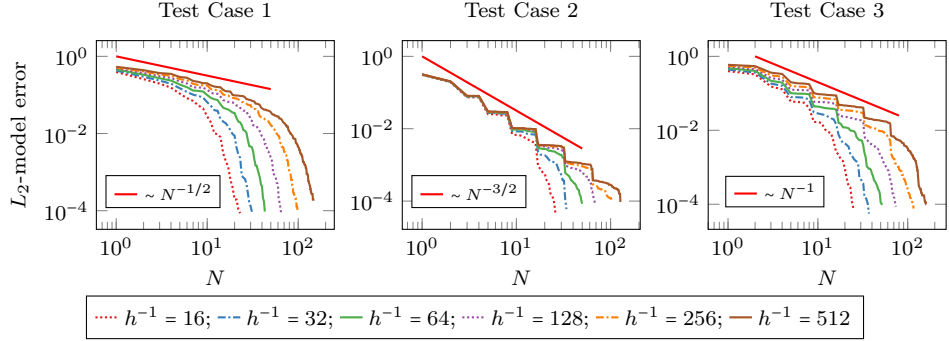

Since we did not derive theoretical convergence results for the greedy algorithm, we investigate the convergence behavior numerically. To that end, we first consider a test case where the best-possible convergence rate of linear approximations is known: 
In \cite{OR16}, it is shown that the Kolmogorov $N$-width of the solution set of Test Case 1 decays with an order of $N^{-1/2}$. 
In the corresponding results of our greedy algorithm, we indeed observe the same (and thus optimal) convergence behavior, see Figure \ref{Fig:RB_results}.

In Test Case 2 we choose constant reaction and source terms that lead to more regular solutions. Here, the greedy algorithm shows a faster convergence of order about $N^{-3/2}$. With discontinuous source and boundary data in Test Case 3 we finally observe an order of roughly $N^{-1}$. 

Similar experiments were also performed in \cite{DaPlWe14}, where reduced models are built by the so-called Double Greedy algorithm that chooses reduced trial spaces and uses additional loops to find stabilized reduced test spaces (of larger dimension).
To realize a fair comparison with our approach, we also implemented a ``strong'' Double Greedy algorithm using the model error instead of a surrogate in \cite[Algorithm 4]{DaPlWe14}. For the full solutions we use the discretization of the \emph{Optimal Test} method described in Section  \ref{subsec:comparison_Dahmen_numerics}.
We then run the ``strong'' variant of the Double Greedy algorithm \cite[Algorithm 5]{DaPlWe14} for Test Case 3 on a training set of 500 equidistant parameter values distributed over $\cP$ and with tolerance $\varepsilon = 0.01$ comparing different thresholds $\beta_{min}$ for the inf-sup stability of the reduced spaces\footnote{In \cite{DaPlWe14} it is proposed to use $\beta_{min} := \zeta \beta_{\delta}$, where $0 < \beta_{\delta} \leq 1$ is a lower bound of the discrete inf-sup constants of the full discretizations for all $\mu \in \cP$ and some $0 < \zeta < 1$, such that the desired threshold is guaranteed to be achievable for all reduced spaces. Here, we simply compare different values of $\beta_{min} < 1$ without computing $\beta_{\delta}$.}. 

The resulting maximum model errors for 500 test parameter values are shown in Figure \ref{Fig:DouGreNErr}. For the smaller stability thresholds of $0.3$ and $0.6$ we observe slight instabilities while for a threshold of $0.7$ the maximum model errors are decreasing for increasing model orders. 
Comparing the approximation properties of the trial spaces of the Double Greedy and Optimal Trial Greedy method, we see that for model orders up to 32 the Double Greedy trial spaces lead to smaller errors than the Optimal Trial spaces of same dimension, while for larger model orders the Optimal Trial reduced spaces perform better.

Since, unlike the new method, for the Double Greedy method the test spaces are significantly larger (for Test Case 3, $\beta_N \geq 0.7$, approximately by a factor of 3) than the trial spaces, the test space dimensions are essential for the online complexity of the reduced saddle point problems.  
In Figure \ref{Fig:DouGreCPU} online computation times for both methods are shown, where we use for the Double Greedy solutions a reformulation of the saddle point problem where the inversion of a test space sized matrix dominates the costs\footnote{Directly solving the larger linear system of size (trial space dim.)+(test space dim.) corresponding to the saddle point formulation leads to comparable results.}. 
We clearly see that the Optimal Trial reduced models outperform the Double Greedy models both when comparing the same trial space dimensions and the same model errors\footnote{Note, however, that as usual online computation times contain only the computation of the coefficients of the reduced solutions in the respective reduced basis. If an assembly of the full-dimensional solution vector is needed, this dominates the costs and is clearly faster for the Double Greedy models, since for the Optimal Trial method the separate parts of the affine decomposition of the trial space have to be assembled, and the trial space vector is usually larger.}.

These results show that for the rather challenging Test Case 3 the Optimal Trial method leads to comparable and for larger model orders even better approximation properties for the same dimension of the trial spaces and to faster online computation times than the Double Greedy method. We note that for smoother cases, e.g.\ Test Case 2, the Optimal Trial models show the same, but not better convergence order than the Double Greedy models.

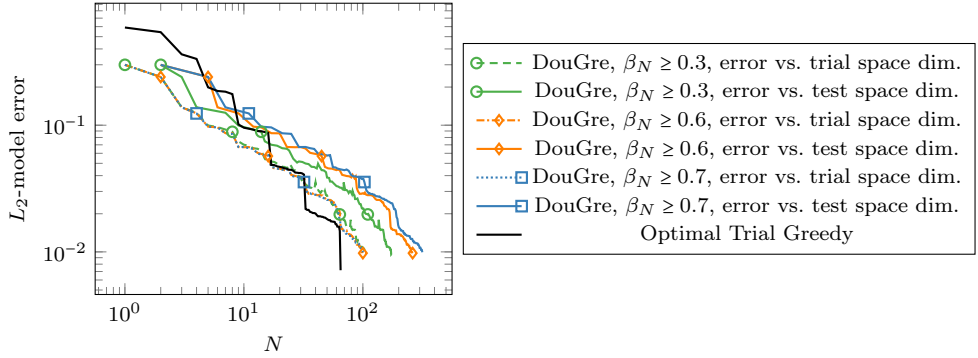
\begin{figure}
\footnotesize{
\begin{tikzpicture} \hypersetup{hidelinks}
	\definecolor{c1}{RGB}{228,26,28}
	\definecolor{c2}{RGB}{55,126,184}
	\definecolor{c3}{RGB}{77,175,74}
	\definecolor{c4}{RGB}{152,78,163}
	\definecolor{c5}{RGB}{255,127,0}
	\definecolor{c6}{RGB}{166,86,40}
	\definecolor{c7}{RGB}{247,129,191}
	\definecolor{c8}{RGB}{153,153,153}
	\begin{loglogaxis}
	[width=180pt, xlabel=$N$,ylabel=$L_2$-model error, legend style={
		at={(1.03,0.5)},
		anchor=west}
	]
\addplot+ [color=c3, thick, densely dashed, mark=o, mark indices={1, 8, 64}, mark options=solid] table [x index={0}, y index={2}, col sep=comma]{RB3_dou_gre_test_model_n_256_test_random_500_beta_min_0.3.csv};
	\addlegendentry{DouGre, $\beta_N \geq 0.3$, error vs.\ trial space dim.};
	\addplot+ [color=c3, thick, solid, mark=o, mark indices={1, 8, 64}, mark options=solid] table [x index={1}, y index={2}, col sep=comma]{RB3_dou_gre_test_model_n_256_test_random_500_beta_min_0.3.csv};
	\addlegendentry{DouGre, $\beta_N \geq 0.3$, error vs.\ test space dim.};
\addplot+ [color=c5, thick, densely dashdotted, mark=diamond, mark indices={2, 16, 100}, mark options=solid] table [x index={0}, y index={2}, col sep=comma]{RB3_dou_gre_test_model_n_256_test_random_500_beta_min_0.6.csv};
	\addlegendentry{DouGre, $\beta_N \geq 0.6$, error vs.\ trial space dim.};
	\addplot+ [color=c5, thick, solid, mark=diamond, mark indices={2, 16, 100,}] table [x index={1}, y index={2}, col sep=comma]{RB3_dou_gre_test_model_n_256_test_random_500_beta_min_0.6.csv};
	\addlegendentry{DouGre, $\beta_N \geq 0.6$, error vs.\ test space dim.};
\addplot+[color=c2, thick, densely dotted, mark=square, mark indices={4, 32, 256}, mark options=solid, legend image post style={mark indices={3}}] table [x index={0}, y index={2}, col sep=comma]{RB3_dou_gre_test_model_n_256_test_random_500_beta_min_0.7.csv};
	\addlegendentry{DouGre, $\beta_N \geq 0.7$, error vs.\ trial space dim.};
	\addplot+[color=c2, thick, solid, mark=square, mark indices={4, 32, 256}, legend image post style={mark indices={3}}] table [x index={1}, y index={2}, col sep=comma]{RB3_dou_gre_test_model_n_256_test_random_500_beta_min_0.7.csv};
	\addlegendentry{DouGre, $\beta_N \geq 0.7$, error vs.\ test space dim.};
\addplot[color=black, thick, solid] table [x index={0}, y index={2}, col sep=comma]{RB3_greedy_timing_n_512_equid_500_n_solves_5000.csv};
		\addlegendentry{Optimal Trial Greedy}
\end{loglogaxis}
\end{tikzpicture}
}
\caption{Test Case 3, $h^{-1} = 512$. Maximum errors of 500 test parameter values for reduced models from Algorithm \ref{alg:strong greedy} (Optimal Trial Greedy) and the strong Double Greedy (DouGre) Algorithm with different lower inf-sup bounds, plots of maximum error versus trial space dimension and test space dimension, respectively. \label{Fig:DouGreNErr}}
\vspace{-1em}
\end{figure}

\begin{figure}
\footnotesize{
\begin{tikzpicture}[scale=0.85] \hypersetup{hidelinks}
	\definecolor{c1}{RGB}{228,26,28}
	\definecolor{c2}{RGB}{55,126,184}
	\definecolor{c3}{RGB}{77,175,74}
	\definecolor{c4}{RGB}{152,78,163}
	\definecolor{c5}{RGB}{255,127,0}
	\definecolor{c6}{RGB}{166,86,40}
	\begin{loglogaxis}
	[width=150pt, xlabel=$N$, ylabel=Computation time (s)]
		\addplot[color=black, thick, solid] table [x index={0}, y index={1}, col sep=comma]{RB3_greedy_timing_n_512_equid_500_n_solves_5000.csv};
%
		\addplot[color=c2, thick, densely dashdotted] table [x index={0}, y index={1}, col sep=comma]{RB3_dou_gre_timing_n_256_beta_min_0.7_n_solves_5000.csv};
	\end{loglogaxis}
\end{tikzpicture}
\begin{tikzpicture}[scale=0.85] \hypersetup{hidelinks}
	\definecolor{c1}{RGB}{228,26,28}
	\definecolor{c2}{RGB}{55,126,184}
	\definecolor{c3}{RGB}{77,175,74}
	\definecolor{c4}{RGB}{152,78,163}
	\definecolor{c5}{RGB}{255,127,0}
	\definecolor{c6}{RGB}{166,86,40}
	\begin{loglogaxis}
	[width=150pt, xlabel=Computation time (s), ylabel=$L_2$-model error,legend style={
		at={(1.03,0.5)},
		anchor=west}]
		\addplot[color=black, thick, solid] table [x index={1}, y index={2}, col sep=comma]{RB3_greedy_timing_n_512_equid_500_n_solves_5000.csv};
		\addlegendentry{Optimal Trial Greedy};
		\addplot[color=c2, thick, densely dashdotted] table [x index={1}, y index={2}, col sep=comma]{RB3_dou_gre_timing_n_256_beta_min_0.7_n_solves_5000.csv};
		\addlegendentry{Double Greedy, $\beta_N \geq 0.7$};
	\end{loglogaxis}
\end{tikzpicture}
\caption{Test Case 3, $h^{-1} = 512$. Comparison of online computation times (median of 5000 runs) for reduced models from Optimal Trial Greedy Algorithm and strong Double Greedy Algorithm with $\beta_N \geq 0.7$. Left: Computation time versus trial space dimension, right: maximum model error versus computation time. \label{Fig:DouGreCPU} }
}
\vspace{-2em} 
\end{figure}
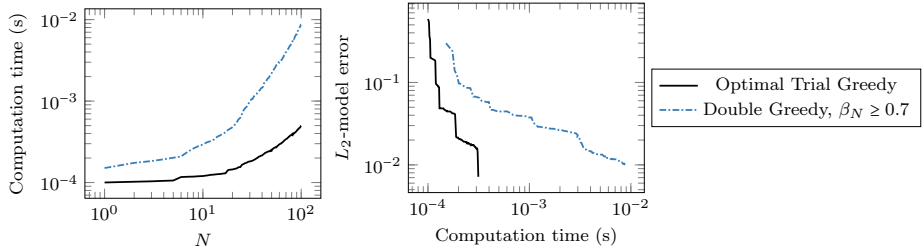

Finally, to test the hierarchical error estimator described in \S \ref{Subsec:RB_Error_Analysis}, we use Test Case 2 with mesh size $\delta = (h,h)$, $h^{-1} = 512$. For the reduced space $Y^N$,  we choose a greedy basis with tolerance $\varepsilon = 10^{-2}$, which here corresponds to $N=13$. For the error estimator reference space $Y^M\supset Y^N$, we compare spaces with tolerances $\varepsilon = 10^{-2.5}, 10^{-3}, 10^{-3.5},$ and $10^{-4}$, leading to $M = 31, 62, 91,$ and 127, respectively. The results in Figure \ref{Fig:Hier} show the quantitative good performance. Note, that the values of $M$ are significantly larger than reported for the hierarchical error estimator in \cite{HORU2018} which is due to the fact that $M$ is determined differently and transport problems are not considered there.
\begin{figure}[tb]
	\pgfplotsset{
		compat=1.11,
		legend image code/.code={
			\draw[mark repeat=2,mark phase=2]
			plot coordinates {
				(0cm,0cm)
				(0.15cm,0cm)        
				(0.3cm,0cm)         
			};%
		}
	}
\begin{center}
	\footnotesize{
	\begin{tikzpicture}[scale=0.85] \hypersetup{hidelinks}
	\definecolor{c1}{RGB}{228,26,28}
	\definecolor{c2}{RGB}{55,126,184}
	\definecolor{c3}{RGB}{77,175,74}
	\definecolor{c4}{RGB}{152,78,163}
	\definecolor{c5}{RGB}{255,127,0}
	\begin{semilogyaxis}[width=150pt, xlabel=$\mu$,ylabel=Model Error, domain=0.2:pi/2-0.2]
		\addplot [color=c1, thick] table [x index={0}, y index={1}, col sep=comma]{rb2_hier_table_n_512_equid_500_test_equid_midpoints_500.csv};
	\end{semilogyaxis}
	\end{tikzpicture}
	\begin{tikzpicture}[scale=0.85]
	\definecolor{c1}{RGB}{228,26,28}
	\definecolor{c2}{RGB}{55,126,184}
	\definecolor{c3}{RGB}{77,175,74}
	\definecolor{c4}{RGB}{152,78,163}
	\definecolor{c5}{RGB}{255,127,0}
	\begin{axis}[width=150pt, xlabel=$\mu$,ylabel=Efficiency, domain=0.2:pi/2-0.2,
	legend entries={$\varepsilon=10^{-2.5}$,$\varepsilon=10^{-3}$,$\varepsilon=10^{-3.5}$,$\varepsilon=10^{-4}$},
	legend style={
		at={(1.03,0.5)},
		anchor=west}
	]
	\addplot [color=c2, thick] table [x index={0}, y index={2}, col sep=comma]{rb2_hier_table_n_512_equid_500_test_equid_midpoints_500.csv};
	\addplot [color=c5, thick] table [x index={0}, y index={3}, col sep=comma]{rb2_hier_table_n_512_equid_500_test_equid_midpoints_500.csv};
	\addplot [color=c3, thick] table [x index={0}, y index={4}, col sep=comma]{rb2_hier_table_n_512_equid_500_test_equid_midpoints_500.csv};
	\addplot [color=c1, thick] table [x index={0}, y index={5}, col sep=comma]{rb2_hier_table_n_512_equid_500_test_equid_midpoints_500.csv};
	\end{axis}
	\end{tikzpicture}
}
\end{center}
	\caption{Test Case 2, $h^{-1} = 512$. Model errors $\|u^N-u^\delta\|_{L_2(\Omega)}$ for all test parameter values (left) and ratios of estimated and real model errors $\|u^N-u^M\|_{L_2(\Omega)}/ \|u^N-u^\delta\|_{L_2(\Omega)}$ (right).
	\label{Fig:Hier}}
\vspace{-1.5em}
\end{figure}
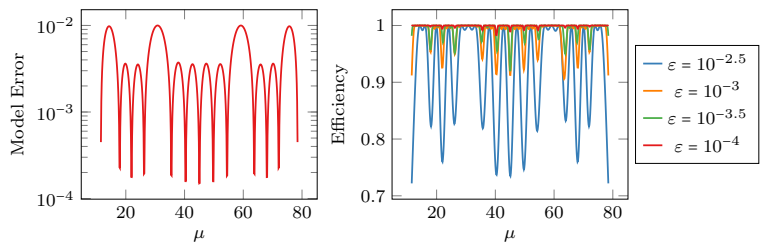

\section{Conclusions}\label{Sec:7}

In this work, we presented a Petrov-Galerkin method for (parametrized) transport equations leading to a computationally feasible optimally stable numerical scheme that is easy to implement. 

Numerical experiments show convergence of order about 1/3 for non-smooth $L_2$-solutions. Despite the $L_2$-framework, higher convergence orders between 1 and 2 can be observed for smooth solutions, even though tensor product discrete spaces may limit the convergence order to 1 due to unphysical restrictions of the trial space at the outflow boundary. 
The proposed method shows similar ratios of errors and computational costs to \cite{DHSW2012}, where fixed trial spaces are used. We thus conclude that our non-standard problem-dependent trial spaces have satisfying approximation properties for the considered test cases. 

Moreover, the framework allows for an efficient realization and implementation of reduced basis methods for parametrized transport equations while ensuring optimal stability for full and reduced spaces.
The suggested (strong) greedy algorithm realizes the convergence order of the Kolmogorov n-width for a non-smooth transport problem.
A comparison with the algorithm in \cite{DaPlWe14} that uses fixed trial spaces and therefore 
needs additional stabilization techniques
shows comparable, or even better convergence rates and significantly lower online costs for the new framework. The results suggest that the new framework might be especially beneficial for problems where a stabilization is rather challenging.

\section*{Acknowledgments}
We are grateful to Gerrit Welper for providing us with additional information to \cite{DaPlWe14}.

\appendix

\section{Proof of Proposition \ref{Prop:Poincare_B1_B2}} \label{Proof:Poincare_B1_B2}

We first give a more formal definition of an $\Omega$-filling flow. To that end, let the flow associated with the vector field $\vec{b} \in C^1(\bar{\Omega})^n$ be described by the integral curves $\xi : (s, x) \in [\sigma_x, \tau_x] \times \bar \Omega \to \xi(s,x) \in \bar \Omega$ that solve
\begin{equation*}
\frac{d \xi}{ds} = \vec{b}(\xi), \quad \xi(0, x) = x.
\end{equation*}

\begin{definition}[{$\Omega$-filling flow, \cite{Azerad1996}}] \label{def:omega_filling_xi}
Let $\vec{b} \in C^1(\bar{\Omega})^n$, then the associated flow is called $\Omega$-filling, if there exists $T > 0$ such that for almost every $x \in \bar \Omega$ there exist $x_0 \in \Gamma_-$ and $0 \leq t \leq T$ such that
\begin{equation*}
x = \xi(t, x_0).
\end{equation*}
\end{definition}

Similar to  \cite[Lem.\ 7]{Azerad1996} we show the following lemma. 
\begin{lemma} \label{lem:rho_var_3}
If the flow associated with $\vec{b}$ is $\Omega$-filling, then there exists $\rho \in L_\infty(\Omega)$ 
such that
\begin{equation} \label{eq:def_rho}
\begin{aligned}
\vec{b} \cdot \nabla\rho = 2 \quad \text{in } \Omega, \quad \text{and} \quad \rho = 0 \quad \text{on } \Gamma_-.
\end{aligned}
\end{equation}
Moreover, we have $\|\rho\|_{L_\infty(\Omega)} \leq 2 T$ and $\rho \geq 0$ almost everywhere in $\Omega$.
\end{lemma}
\begin{proof}
The function $\rho$ can be found by the method of characteristics: Since the flow associated with $\vec{b}$ is $\Omega$-filling, for almost every $x \in \Omega$, there exist $x_0 \in \Gamma_-$ and $0 \leq t \leq T$ with $x = \xi(t, x_0)$. Define $\rho(x) = 2t$. Since $0 \leq t \leq T$, we get $\rho \in L_\infty(\Omega)$\footnote{$\rho$ is in general not continuous: Consider e.g.\ a non-convex domain $\Omega$ where a characteristic curve is tangential to the boundary at some (isolated) $x \in \Gamma_0$, but not in a neighborhood of $x$. Then $\rho$ is discontinuous along the characteristic curve starting from $x$.}, 
$\|\rho\|_{L_\infty(\Omega)} \leq 2T$, and $\rho \geq 0$ almost everywhere in $\Omega$.
By definition, for $x_0 \in \Gamma_-$ we have $\xi(0, x_0) = x_0$, i.e.\ $\rho(x_0) = 0$, which means $\rho|_{\Gamma_-} = 0$.  
Furthermore it holds for almost every $x \in \Omega$
\begin{equation*}
\begin{aligned}
\vec{b(x)} \cdot \nabla \rho(x) 
&= \vec{b}(\xi(t, x_0)) \cdot \nabla \rho(\xi(t, x_0)) 
= \frac{d}{dt} \xi(t, x_0) \cdot \nabla \rho(\xi(t, x_0)) \\
&= \frac{d}{dt} \rho(\xi(t, x_0)) 
= \frac{d}{dt} 2t 
= 2,
\end{aligned}
\end{equation*} 
i.e., $\rho$ fulfills \eqref{eq:def_rho}.
\end{proof}

With these preliminaries, we can now give the proof of Proposition \ref{Prop:Poincare_B1_B2}.

\begin{proof}[Proof of Proposition \ref{Prop:Poincare_B1_B2}] 
We first show \eqref{eq:curved_Poincare}, i.e., $\|v\|_{L_2(\Omega)} \leq C \|B^*_0 v\|_{L_2(\Omega)}$. 
Let thus $v \in C^1_{\Gamma_+}(\Omega)$. If condition (i) holds, we can slightly adapt the proof of \cite[Thm.\ 1]{AzePou96}: Let $\rho$ be given as in Lemma \ref{lem:rho_var_3}. Then,  
\begin{align*}
(B^*_{\circ} v, \rho v)_{L_2(\Omega)} 
&= (-\vec{b} \cdot \nabla v + v(c - \nabla \cdot \vec{b}), \rho v)_{L_2(\Omega)} \\
&= - \int_{\Omega} \vec{b} \cdot \nabla v \rho v dx + \int_{\Omega} v^2 \rho (c - \nabla \cdot \vec{b}) dx \\
&= - \int_{\Omega} \tfrac{1}{2} \rho \vec{b} \cdot \nabla v^2 dx + \int_{\Omega} v^2 \rho (c - \nabla \cdot \vec{b}) dx \\
&= \int_{\Omega} \tfrac{1}{2} \nabla \cdot (\rho \vec{b}) v^2 dx + \int_{\Omega} v^2 \rho (c - \nabla \cdot \vec{b}) dx,
\end{align*}
where we have no boundary integral from the partial integration since the traces of $v$ on $\Gamma_+$ and of $\rho$ on $\Gamma_-$ vanish. Further we obtain
\begin{equation}\label{eq:B_adjvrhov}
(B^*_{\circ} v, \rho v)_{L_2(\Omega)} 
= \int_{\Omega} v^2  (\tfrac{1}{2} \underbrace{\vec{b} \cdot \nabla \rho}_{=2} + \underbrace{\rho}_{\geq 0} \underbrace{ (c - \tfrac{1}{2} \nabla \cdot \vec{b} )}_{\geq 0}  ) dx 
\geq \|v\|_{L^2(\Omega)}^2.
\end{equation}
Using $\|\rho v\|_{L_2(\Omega)} \leq \|\rho\|_{L_\infty(\Omega)} \|v\|_{L_2(\Omega)} \leq 2T \|v\|_{L_2(\Omega)}$ we have
\begin{equation}  \label{eq:Badjpoinc_1}
\|B^*_\circ v\|_{L_2(\Omega)} \geq \|\rho v\|_{L_2(\Omega)}^{-1} (B^*_\circ v, \rho v)_{L_2(\Omega)} \geq \frac{1}{2T} \|v\|_{L_2(\Omega)}. 
\end{equation}
For condition (ii), i.e., $c - \tfrac{1}{2} \nabla \cdot \vec{b} \geq \kappa > 0$, we obtain by integration by parts (see \cite[Lem.\ 3.1.1]{Welper2013})
\begin{align} \nonumber
(B^*_{\circ} v, v)_{L_2(\Omega)}
&= \int_{\Omega} -v \vec{b} \cdot \nabla v dx 
+ \int_{\Omega} v^2 (c-\nabla \cdot b) dx \\\nonumber
&= - \tfrac{1}{2} \int_{\Omega} v \vec{b} \cdot \nabla v dx
+ \tfrac{1}{2} \int_{\Omega} v \vec{b} \cdot \nabla v + v^2 \nabla \cdot \vec{b}  dx
- \tfrac{1}{2} \int_{\Gamma_-} v^2 \vec{b} \cdot \vec{n} ds \\\nonumber
&\quad+ \int_{\Omega} v^2 (c-\nabla \cdot b) dx \\
&= \int_{\Omega} v^2 (c- \tfrac{1}{2}\nabla \cdot b) dx 
- \tfrac{1}{2} \int_{\Gamma_-} v^2 \underbrace{\vec{b} \cdot \vec{n}}_{< 0} ds 
\geq \kappa \|v\|_{L^2(\Omega)}^2 \label{eq:est_aadj_v_v}
\end{align}
and thus 
\begin{equation*}
\|B^*_{\circ} v\|_{L_2(\Omega)} \geq \kappa \|v\|_{L^2(\Omega)},
\end{equation*}
i.e., \eqref{eq:curved_Poincare} holds for both cases.

Since \eqref{eq:curved_Poincare} implies injectivity of $B^*_\circ$ on $C^1_{\Gamma_+}(\Omega)$, which is dense in $L_2(\Omega)$, Assumption (B1) is fulfilled.

To prove Assumption (B2), we slightly modify the proof of \cite[Thm.\ 16]{Azerad1996}. To prove density of $\ran(B^*_\circ)$ in $L_2(\Omega)$, we take $w \in L_2(\Omega)$ that is orthogonal to $\ran(B^*_\circ)$ and show $w \equiv 0$. We thus have
\begin{equation*}
(B^*_\circ v, w)_{L_2(\Omega)} = 0 \quad \forall v \in C^1_{\Gamma_+}(\Omega).
\end{equation*}
Let at first $v \in C^1_0(\Omega)$. We then have
\begin{align} \label{eq:surj_pi}
0 
= \int_{\Omega} - \vec{b} \cdot \nabla v w  + (c - \nabla \cdot \vec{b}) v w dx 
= \int_{\Omega} - \nabla \cdot (\vec{b}v) w + cvw dx
\end{align}
By partial integration we see that $\vec{b} \cdot \nabla w + cw$ is a distribution of order 1 with
\begin{equation*}
\langle \vec{b} \cdot \nabla w + cw, v \rangle = 0,
\end{equation*}
which already means $\vec{b} \cdot \nabla w + cw  = 0$, i.e., $\vec{b} \cdot \nabla w  = - cw \in L_2(\Omega)$. 
Therefore, since $w \in H(\Omega, \vec{b}) := \{\phi \in L_2(\Omega): \vec{b} \cdot \nabla \phi \in L_2(\Omega) \}$ and $\partial \Omega$ is piecewise $C^1$, we can consider the trace $w|_{\Gamma_-} \in L_{2,loc}(\Gamma_-, |\vec{b} \cdot \vec{n}|)$ (see \cite[Prop.\ I.1]{GeyLey1987}). Let now $v \in C^1_{\Gamma_+}(\Omega)$. We then obtain from partial integration of \eqref{eq:surj_pi}, using $\vec{b} \cdot \nabla w + cw  = 0$ and $v|_{\Gamma_+} = 0$ that
\begin{equation*}
\int_{\Gamma_-} v w \vec{b} \cdot \vec{n} ds = 0.
\end{equation*}
Since $v$ is arbitrary on $\Gamma_-$ and $\vec{b} \cdot \vec{n} < 0$ on $\Gamma_-$ we thus have $w|_{\Gamma_-} = 0$. 

We now consider the curved Poincar\'{e} inequality \eqref{eq:curved_Poincare} for the (non-adjoint) operator $B_\circ z = \vec{b} \cdot \nabla z + cz$: By setting $\vec{b}_1 = -\vec{b}$ and $c_1 = c - \nabla \cdot \vec{b}$, \eqref{eq:curved_Poincare} reads
\begin{equation} \label{eq:curved_Poincare_dual}
\|- \vec{b}_1 \cdot \nabla z + (c_1 - \nabla \cdot \vec{b}_1)z\|_{L_2(\Omega)}
= \|\vec{b} \cdot \nabla z + cz \|_{L_2(\Omega)} 
\geq C_1\|z\|_{L_2(\Omega)} 
\quad \forall z \in C^1_{\Gamma_-}(\Omega),
\end{equation}
as $\Gamma_-$ is the outflow boundary for $\vec{b}_1 = -\vec{b}$. Since $C^1_{\Gamma_-}(\Omega)$ is dense in $\{\phi \in L_2(\Omega): \vec{b} \cdot \nabla \phi \in L_2(\Omega), \phi|_{\Gamma_-} = 0\}$, \eqref{eq:curved_Poincare_dual} we obtain $0 = \|\vec{b} \cdot \nabla w + cw \|_{L_2(\Omega)} \geq C_1 \|w\|_{L_2(\Omega)}$, and thus $w=0$.
Hence, also (B2) is fulfilled.  
\end{proof}

\section{Examples for conditions leading to well-posedness of the variational formulation}
\label{sec:counterexamples}
In Proposition \ref{Prop:Poincare_B1_B2} we give conditions on the data functions such that the corresponding operator fulfills Assumption \ref{lemma:B12}.

Considering a bounded, polyhedral domain $\Omega \subset \R^n$, $n>1$ with Lipschitz boundary that consists of finitely many polyhedral faces again having Lipschitz boundaries, the authors of \cite{DHSW2012} albeit claim in Remark 2.2(i) that the assumption $0 \neq \vec{b} \in C^1(\Omega)^{n}$ is already sufficient for Assumption \ref{lemma:B12}.
Here, we want to give counterexamples to that claim showing that the more stringent condition of an $\Omega$-filling flow given in condition (i) of Proposition \ref{Prop:Poincare_B1_B2} is indeed necessary.

To that end, we consider $\Omega \subset \R^2$ with advection field $\vec{b}(x,y) = (-y, x)$ and no reaction $c \equiv 0$. It holds $\nabla \cdot \vec{b} = 0$, thus, the adjoint operator is simply $B^*_\circ v = -\vec{b} \cdot \nabla v$.

The easiest example is an annular domain 
\begin{equation*}
\Omega_1 = \{(x,y) \in \R^2: 0.25 < x^2 + y^2 < 1 \} 
\end{equation*}
(which is however not polyhedral, see Figure \ref{Im:2D_domains}, left). It holds $\vec{b} \neq 0$ on $\overline{\Omega}_1$. The boundary has the form of two circles: $\Gamma = \partial \Omega_1 = \{(x,y) \in \R^2: x^2 + y^2 = 0.25 \} \cup \{(x,y) \in \R^2: x^2 + y^2 = 1 \}$, outward normal is $\vec{n} = (x,y)$ on $\{x^2+y^2 = 1\}$ and $\vec{n} = -2(x,y)$ on $\{x^2+y^2=0.25\}$. Since $\vec{b} \cdot \vec{n} = (-y, x) \cdot C (x,y) = C(-xy + xy) = 0$ for a constant $C \in \mathbb{R}$, $C\neq 0$, the whole boundary belongs to $\Gamma_0$. Therefore, $v \equiv 1 \in C^1_{\Gamma_+}(\Omega_1)$ but $\vec{b} \cdot \nabla v = 0$, i.e.\ $B^*_\circ$ is not injective on $C^1_{\Gamma_+}(\Omega_1)$.

\begin{figure}
\includegraphics[width=0.8\textwidth]{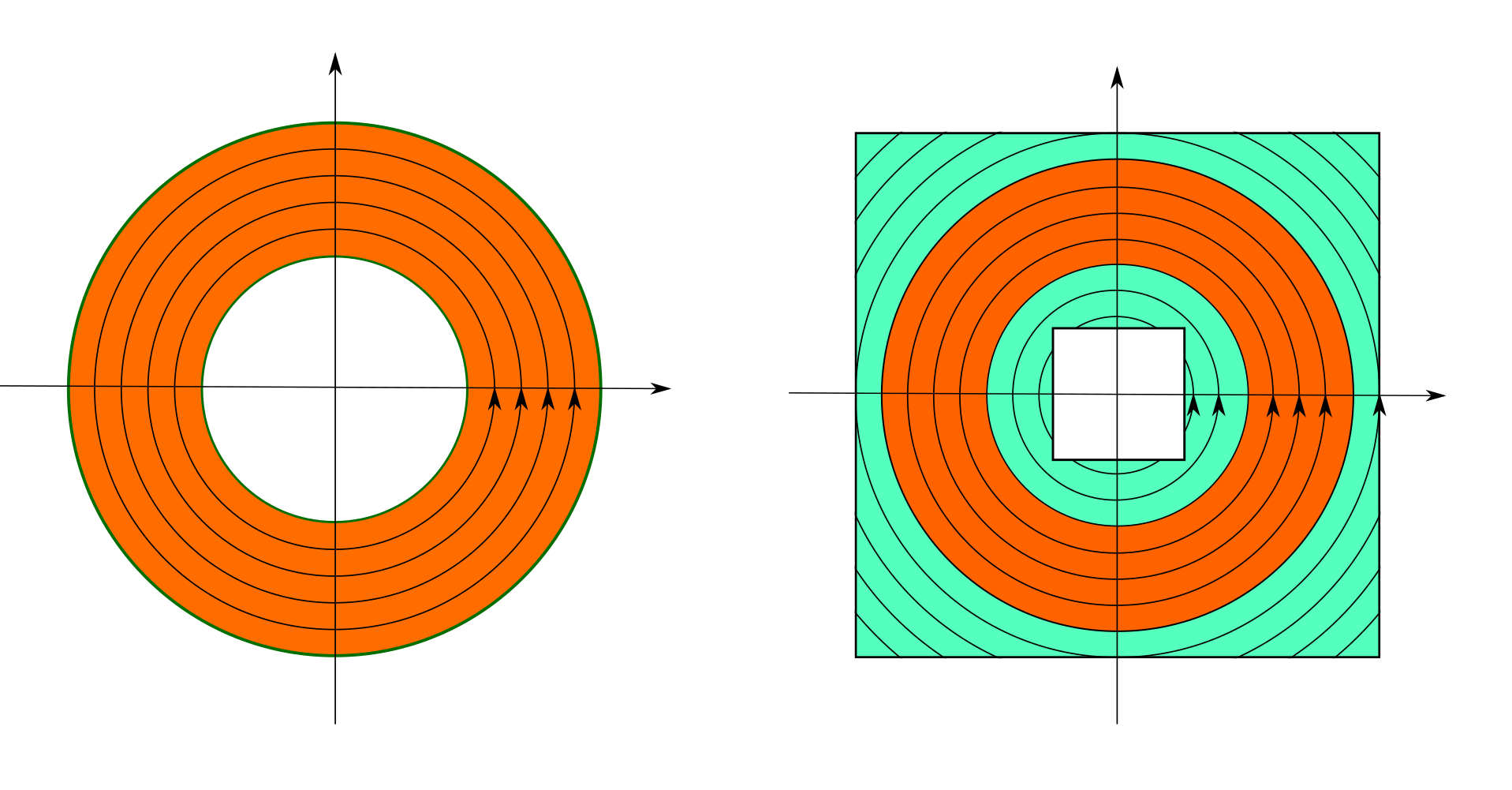}
\caption{Domains $\Omega_1$ (left) and $\Omega_2$ with characteristic curves and $\supp u$ in orange. \label{Im:2D_domains}}
\end{figure}

Even with a polyhedral domain with polygonal boundary and $\Gamma_0 \subsetneq \partial \Omega$ the problem may not be well-posed: Consider to that end
\begin{equation*}
\Omega_2 = (-1,1)^2 \setminus (-0.25, 0.25)^2
\end{equation*}
(see Figure \ref{Im:2D_domains}, right) with $\vec{b}$ and $c$ as before. We thus have again $\vec{b} \neq 0$ on $\overline{\Omega}_2$. Then, let $0 \neq \psi \in C^1([0,1])$ with $\supp \psi \subset [0.5, 0.9]$ and consider
\begin{equation*}
u(x,y) = \psi(\sqrt{x^2 + y^2}).
\end{equation*}
With this definition, $u|_{\partial\Omega_2} = 0$, i.e., $u \in C^1_{\Gamma_+}(\Omega_2)$. The characteristic curves of $\vec{b}$ are circle-shaped of the form $\gamma(t) = r(\cos(\phi + t), \sin(\phi + t))$ for a starting point $(x,y) = r(\cos \phi, \sin \phi)$. The rotational invariant function $u$ is thus constant on the characteristic curves, therefore it holds $\vec{b} \cdot \nabla u = 0$, and $B^*_\circ$ is again not injective.

\section{A (strong in time) space-time variational formulation}
\label{sec:strong_in_time}
As an alternative to our approach of an ultraweak variational form in space \emph{and} time described in Section \ref{Sec:2}, one could also take the point of view of using an ultraweak variational formulation \emph{in space only} and keep the first order derivative in time (i.e., not using integration by parts in time). Integrating over time then results in a space-time framework requiring more regularity in time (let us call it ``strong''). 
In order to fix notation, we first interpret the time-independent problem \eqref{eq:stationary} as an operator equation $A_\circ u := \vec{b}\cdot \nabla u +c\, u = f_\circ$ in some function space $V'$, where $V\hookrightarrow L_2(D)\hookrightarrow V'$ is a Gelfand triple (and $V$ is the $L_2$-dual of $V'$), i.e., $A_\circ:H\to V'$. Accordingly, $B_\circ u = \dot{u} + A_\circ u=f_\circ$ is seen as an equation pointwise in $V'$ for $t\in I$, which means we can multiply \eqref{eq:TE} with some smooth test function $C^0(\bar{I}; V)$ and integrate over time: 
\begin{align}\label{eq:def_b}
	b(w,v) 
	&:= \int_0^T \la \dot w(t) , v(t) \ra_{V' \times V} dt + \int_0^T a( w(t) , v(t)) dt \\
	&= f(v) := \int_0^T \la f_\circ(t),v(t) \ra_{V' \times V} dt. \nonumber
\end{align}
We get that $b:\mathcal{X} \times  \mathcal{Y}\to\R$ with the trial space
\[
	\mathcal{X} := \{ v \in L_2(I;H): \,\dot v \in L_2(I;V'),\, v(0)=0 \}
	= L_2(I;H) \cap H^1_{(0)}(I;V')
\]
and the test space $\mathcal{Y} := L_2(I;V)$. Here, $H^1_{(0)}(I;V') := \{ v \in H^1(I;V'):\, v(0)=0 \}$\footnote{Note, that $v(0)\in V'$, since $H^1(I;X) \hookrightarrow C(\bar I;X)$ for any normed linear space $X$.} equipped with the standard graph norm  $\|v\|_{\mathcal{X}} :=  ( \|v\|_{L_2(I;H)}^2 + \|\dot v\|_{L_2(I;V')}^2  )^{1/2}$ for $v \in \mathcal{X}$. Finally, the norm in $\mathcal{Y}$ is $\|\cdot\|_{\mathcal{Y}}\equiv \|\cdot\|_{L_2(I;V)}$. This means that $g$ can be chosen in $\mathcal{Y}' \cong L_2(I;V')$. 
This results in the variational formulation:
\begin{equation}
	\text{Find } u \in \mathcal{X}: \quad b(u,v) = f(v) \quad \forall \, v \in \mathcal{Y}.
\end{equation}
Since we do not perform integration by parts w.r.t.\ time here, we require $H^1$-regularity in time, which is the reason why we call this formulation \emph{strong in time}.

In order to determine the inf-sup constant of $b(\cdot,\cdot)$ w.r.t.\ the above pair $\cX$, $\cY$, we are going to consider the supremizer $s_u\in\cY$ for some given $0\not= u\in\cX$, which is the solution of the problem $(s_u,v)_\cY = b(u,v) = \langle \dot{u}+Au,v\rangle_{\cY'\times\cY} = (R_\cY^{-1}(\dot{u}+Au),v)_\cY$ for all $v\in\cY$, where $R_\cY:\cY\to\cY'$ is the Riesz operator of $\cY$ and $\langle\cdot,\cdot\rangle_{\cY'\times\cY}$ denotes the dual pairing of $\cY'$ and $\cY$. This means, $s_u = R_\cY^{-1}(\dot{u}+Au)$ and we obtain
\begin{align*}
	\bigg( \sup_{v\in\cY} \frac{b(u,v)}{\| v\|_\cY} \bigg)^2
	&= \| s_u\|_\cY^2
	= \| R_\cY^{-1}(\dot{u}+Au)\|_\cY^2
	= \| \dot{u}\|_{\cY'}^2 + \| Au\|_\cY^2 + 2\, (\dot{u}, Au)_{\cY'}.
\end{align*}
The first two terms can be estimated from above and from below by $\| u\|_\cX^2$, which is exactly what we need. For parabolic problems, the operator $A$ is symmetric and this was used in \cite{MR2891112,MR3194123} to express $( \dot{u}, Au)_{\cY'}$ in terms of the norm of the final time contribution $u(T)$. This is the key to derive optimal inf-sup and continuity constants for parabolic problems.

For transport problems, however, $A$ is \emph{not} symmetric. Defining the symmetric and anti-symmetric part of $A$ as usual, i.e., $A_{\text{sym}}:=\frac12(A+A^*)$, $A_{\text{asy}}:=\frac12(A-A^*)$ we get $A=A_{\text{sym}}+A_{\text{asy}}$, $A_{\text{sym}}^*=A_{\text{sym}}$, $A_{\text{asy}}^*=-A_{\text{asy}}$. We obtain by $u(0)=0$ and the fundamental theorem of calculus 
$	2\, (\dot{u}, Au)_{\cY'}
	= 2\, (\dot{u}, A_{\text{sym}} u)_{\cY'}
		+ 2\,  (\dot{u}, A_{\text{asy}} u)_{\cY'}
	= \| A_{\text{sym}}^{1/2} u(T)\|_{V'}^2
		+ 2\,  (\dot{u}, A_{\text{asy}} u)_{\cY'}
$.
Of course $\| A_{\text{sym}}^{1/2} u(T)\|_{V'}\ge 0$, so that this contribution is no problem. The second part, however, may very well be negative since  $\frac{d}{dt} ({u}(t), A_{\text{asy}} u(t))_{V'}=0$. 
Using H\"older-type estimates, it is not difficult to show the estimate 
$|2\,  (\dot{u}, A_{\text{asy}} u)_{\cY'} | \le T\, \| A_{\text{asy}}\| \sqrt{2} \| \dot{u}\|_{V'}^2$, 
which results in 
$$
	\inf_{u\in\cX} \sup_{v\in\cY}
	\frac{b(u,v)}{\|u\|_\cX\, \| v\|_\cY}
	\ge \min\bigg\{ \beta_a, \sqrt{ 1-T\, \sqrt{2}\| A_{\text{asy}}\|}\bigg\},
$$
where $\beta_a$ denotes the inf-sup constant of the spatial operator $A$. Obviously, this estimate is only meaningful for small final times $T$.

\section{Proof of Lemma \ref{lemma:compactness solution set}}
\label{proof:compactness solution set}
\begin{proof}[Proof of Lemma \ref{lemma:compactness solution set}]
Let $u_{n}(\mu_{n})$ form a sequence in $\mathcal{M}$. Thanks to \eqref{eq:stability of solution}, \eqref{eq:affine decomposition}, and the assumption that $\theta^{q}_{f} \in C^{0}(\bar{\mathcal{P}})$, $q=1,\dots,Q_{f}$, there exists a subsequence $u_{n_{k}}(\mu_{n_{k}}) \in \mathcal{M}$ that converges weakly in $L_{2}(\Omega)$ to a limit $\tilde{u} \in L_{2}(\Omega)$. To infer compactness of $\mathcal{M}$, it thus remains to show that $\tilde{u} \in \mathcal{M}$. To that end, we employ the parameter values $\mu_{n_{k}}$ of the weakly converging subsequence $u_{n_{k}}(\mu_{n_{k}})$ to define a sequence $(\mu_{n_{k}})_{k}$ in $\mathcal{P}$. Thanks to the compactness of $\mathcal{P}$ this sequence has a weakly converging subsequence which we denote w.l.o.g.\ again by $(\mu_{n_{k}})_{k}$ that converges to a limit $\bar{\mu} \in \mathcal{P}$. 

To show continuity of the mappings $\mu \mapsto B_{\mu}^{*}$ and $\mu \mapsto f_{\mu}$, we first note that we have for all $\mu \in \mathcal{P}$ and all $v \in \bar{\mathcal{Y}}$ that
\begin{equation*}
\| B_{\mu}^{*}v \|_{L_{2}(\Omega)} = \|v\|_{\mathcal{Y}_{\mu}} \leq \| v\|_{\bar{\mathcal{Y}}} \quad \text{and} \quad 
\sup_{v \in \bar{\mathcal{Y}}} \frac{| f_{\mu}(v)|}{\| v\|_{\bar{\mathcal{Y}}}} \leq \sup_{v \in \mathcal{Y}_{\mu}} \frac{| f_{\mu}(v) |}{\|v\|_{\mathcal{Y}_{\mu}}} = \|f\|_{\mathcal{Y_{\mu}}'}
\end{equation*}
and thus $B_{\mu}^{*} \in L(\bar{\mathcal{Y}},L_{2}(\Omega))$ and $f_{\mu} \in \bar{\mathcal{Y}}'$. Thanks to the assumption that $B_{\mu}^{*}$ and $f_{\mu}$ are affine w.r.t.\ parameter we may thus infer as in \cite{DaPlWe14} that for all $\mu_{1}, \mu_{2} \in \mathcal{P}$ and all $v \in \bar{\mathcal{Y}}$ we have
\begin{align*}
	\| (B_{\mu_{1}}^{*} - B_{\mu_{2}}^{*})v\|_{L_{2}(\Omega)} &\leq C_{B} \max_{q=1,\dots, Q_{b}} | \theta^{q}_{b}(\mu_{1}) - \theta^{q}_{b}(\mu_{2}) | \enspace \|v \|_{\bar{\mathcal{Y}}}, \\
| f_{\mu_{1}}(v) - f_{\mu_{2}}(v) | &\leq C_{f} \max_{q=1,\dots, Q_{f}} | \theta^{q}_{f}(\mu_{1}) - \theta^{q}_{f}(\mu_{2}) | \enspace \|v \|_{\bar{\mathcal{Y}}},
\end{align*}
which yields the continuity of the mappings $\mathcal{P} \rightarrow L(\bar{\mathcal{Y}},L_{2}(\Omega))$, $\mu \mapsto B_{\mu}^{*}$ and $\mathcal{P} \rightarrow \bar{\mathcal{Y}}'$, $f_{\mu} \in \bar{\mathcal{Y}}'$. As a consequence we have that for all $v \in \bar{\mathcal{Y}}$ the sequences $(B_{\mu_{n_{k}}}^{*}v) \in L_{2}(\Omega)$ and $f_{\mu_{n_{k}}}(v) \in \mathbb{R}$ converge in the following sense
\begin{equation}\label{eq:strong convergence}
\| (B_{\mu_{n_{k}}}^{*} - B_{\bar{\mu}})v\|_{L_{2}(\Omega)} \rightarrow 0 \quad \text{and} \quad |f_{\mu_{n_{k}}}(v) - f_{\bar{\mu}}(v) | \rightarrow 0 \quad \text{for} \enspace \mu_{n_{k}} \rightarrow \bar{\mu}.
\end{equation}
In particular, the sequence $(B_{\mu_{n_{k}}}^{*}v)$ hence converges strongly to $B_{\bar{\mu}}v$ in $L_{2}(\Omega)$. 

We may thus infer that we have for all $v \in \bar{\mathcal{Y}}$ that
\begin{align*}
	&(u_{n_{k}}(\mu_{n_{k}}), B^{*}_{\mu_{n_{k}}}v)_{L_{2}(\Omega)} - f_{\mu_{n_{k}}}(v) \longrightarrow (\tilde{u}, B^{*}_{\bar{\mu}}v)_{L_{2}(\Omega)} - f_{\bar{\mu}}(v)
\end{align*}
and as a consequence $(\tilde{u}, B^{*}_{\bar{\mu}}v)_{L_{2}(\Omega)} = f_{\bar{\mu}}(v)$ for all $v \in \bar{\mathcal{Y}}$. 
To conclude, it remains to prove that there holds
$(\tilde{u}, B^{*}_{\bar{\mu}}v)_{L_{2}(\Omega)} = f_{\bar{\mu}}(v)$ for all $v \in \mathcal{Y}_{\bar{\mu}}$. 
To that end, consider an arbitrary function $v \in \mathcal{Y}_{\bar{\mu}}$. As $\bar{\mathcal{Y}}$ is dense in $\mathcal{Y}_{\bar{\mu}}$, there exists a sequence $v_{n}$ such that $\| v_{n} - v \|_{\mathcal{Y}_{\bar{\mu}}} \rightarrow 0$. Then, we have
\begin{align*}
(\tilde{u}, B^{*}_{\bar{\mu}}v)_{L_{2}(\Omega)} - f_{\bar{\mu}}(v) &= (\tilde{u}, B^{*}_{\bar{\mu}}(v - v_{n}))_{L_{2}(\Omega)} - f_{\bar{\mu}}(v - v_{n}) \\
&\leq \|\tilde{u}\|_{L_{2}(\Omega)} \|B_{\bar{\mu}}^{*}\|_{L(\mathcal{Y}_{\bar{\mu}},L_{2}(\Omega))} \|v-v_{n}\|_{\mathcal{Y}_{\bar{\mu}}} + \|f_{\bar{\mu}}\|_{\mathcal{Y}_{\bar{\mu}}'} \|v-v_{n}\|_{\mathcal{Y}_{\bar{\mu}}}\\ &\longrightarrow 0.
\end{align*} 
We may thus infer that $\tilde{u} = u_{\bar{\mu}} \in \mathcal{M}$, which was to be proven.  
\end{proof}


\begin{thebibliography}{10}
\bibitem{Azerad1996}
{\sc P.~Az{\'e}rad}, {\em Analyse des {\'e}quations de Navier-Stokes en bassin
  peu profond et de l'{\'e}quation de transport}, PhD thesis, Universit{\'e} de
  Neuchatel, 1996.

\bibitem{AzePou96}
{\sc P.~Az{\'e}rad and J.~Pousin}, {\em In{\'e}galit{\'e} de {P}oincar{\'e}
  courbe pour le traitement variationnel de l'{\'e}quation de transport},
  Comptes rendus de l'Acad{\'e}mie des sciences. S{\'e}rie 1, Math{\'e}matique,
  322 (1996), pp.~721--727.

\bibitem{BMNP04}
{\sc M.~Barrault, Y.~Maday, N.~Nguyen, and A.~Patera}, {\em An 'empirical
  interpolation' method: application to efficient reduced-basis discretization
  of partial differential equations}, C. R. Math. Acad. Sci. {P}aris {S}eries
  {I}, 339 (2004), pp.~667--672.

\bibitem{BrDaSt17}
{\sc D.~Broersen, W.~Dahmen, and R.~P. Stevenson}, {\em On the stability of
  {DPG} formulations of transport equations}, Math. Comp., 87 (2018),
  pp.~1051--1082.

\bibitem{Bru2018zenodoversion3}
{\sc J.~Brunken}, {\em Source code to {``}({P}arametrized) first order
  transport equations: {R}ealization of optimally stable {P}etrov-{G}alerkin
  methods{''}}, Sept. 2018.

\bibitem{BTDeGh13}
{\sc T.~Bui-Thanh, L.~Demkowicz, and O.~Ghattas}, {\em Constructively
  well-posed approximation methods with unity inf-sup and continuity constants
  for partial differential equations}, Math. Comp., 82 (2013), pp.~1923--1952.

\bibitem{DHSW2012}
{\sc W.~Dahmen, C.~Huang, C.~Schwab, and G.~Welper}, {\em Adaptive
  {P}etrov-{G}alerkin methods for first order transport equations}, SIAM J.\
  Numer.\ Anal., 50 (2012), pp.~2420--2445.

\bibitem{DaPlWe14}
{\sc W.~Dahmen, C.~Plesken, and G.~Welper}, {\em Double greedy algorithms:
  reduced basis methods for transport dominated problems}, ESAIM Math. Model.
  Numer. Anal., 48 (2014), pp.~623--663.

\bibitem{MR2407757}
{\sc J.~L.~R. d'Alembert}, {\em Textes de math\'ematiques pures (1745--1752)},
  vol.~4 of Complete Works of d'Alembert. Series I. Mathematical Treatises and
  Papers, 1736--1756, CNRS \'Editions, Paris, 2007.

\bibitem{DemGop10}
{\sc L.~Demkowicz and J.~Gopalakrishnan}, {\em A class of discontinuous
  {P}etrov-{G}alerkin methods. {P}art {I}: the transport equation}, Comput.
  Methods Appl. Mech. Engrg., 199 (2010), pp.~1558--1572.

\bibitem{DemGop11}
\leavevmode\vrule height 2pt depth -1.6pt width 23pt, {\em A class of
  discontinuous {P}etrov-{G}alerkin methods. {II}. {O}ptimal test functions},
  Numer. Methods Partial Differential Equations, 27 (2011), pp.~70--105.

\bibitem{DFW16}
{\sc W.~D{\"{o}}rfler, S.~Findeisen, and C.~Wieners}, {\em Space-time
  discontinuous {G}alerkin discretizations for linear first-order hyperbolic
  evolution systems}, Comput. Meth. in Appl. Math., 16 (2016), pp.~409--428.

\bibitem{EHKS2015}
{\sc C.~Engwer, T.~Hillen, M.~Knappitsch, and C.~Surulescu}, {\em Glioma follow
  white matter tracts: a multiscale {DTI}-based model}, J.\ Math.\ Biol., 71
  (2015), pp.~551--582.

\bibitem{GeyLey1987}
{\sc G.~Geymonat and P.~Leyland}, {\em Transport and propagation of a
  perturbation of a flow of a compressible fluid in a bounded region}, Archive
  for Rational Mechanics and Analysis, 100 (1987), pp.~53--81.

\bibitem{RB:Wave}
{\sc S.~Glas, A.~Patera, and K.~Urban}, {\em Reduced basis methods for the wave
  equation}.
\newblock Unpublished manuscript, 2018.

\bibitem{Haasdonk:RB}
{\sc B.~Haasdonk}, {\em Reduced basis methods for parametrized {PDE}s -- a
  tutorial}, in Model Reduction and Approximation, P.~Benner, A.~Cohen,
  M.~Ohlberger, and K.~Willcox, eds., SIAM, Philadelphia, 2017, ch.~2,
  pp.~65--136.

\bibitem{HORU2018}
{\sc S.~{Hain}, M.~{Ohlberger}, M.~{Radic}, and K.~{Urban}}, {\em {A
  Hierarchical A-Posteriori Error Estimator for the Reduced Basis Method}},
  Feb. 2018.

\bibitem{HeRoSt16}
{\sc J.~Hesthaven, G.~Rozza, and B.~Stamm}, {\em Certified reduced basis
  methods for parametrized partial differential equations}, Springer Briefs in
  Mathematics, Springer, Cham, 2016.

\bibitem{Hillen2006}
{\sc T.~Hillen}, {\em M5 mesoscopic and macroscopic models for mesenchymal
  motion}, J.\ Math.\ Biol., 53 (2006), pp.~585--616.

\bibitem{OR16}
{\sc M.~Ohlberger and S.~Rave}, {\em Reduced basis methods: Success,
  limitations and future challenges}, in Proceedings of {ALGORITMY 2016}, 20th
  Conference on Scientific Computing, March 13-18, 2016, {Handlovi{\v{c}}ova
  A., and Sev{\v{c}}ovi{\v{c}}, D.}, ed., Vysoke Tatry, Podbanske, Slovakia,
  2016, Publishing House of Slovak University of Technology in Bratislava,
  pp.~1--12.

\bibitem{QiuShu2005}
{\sc J.~Qiu and C.-W. Shu}, {\em A comparison of troubled-cell indicators for
  {R}unge--{K}utta discontinuous {G}alerkin methods using weighted essentially
  nonoscillatory limiters}, SIAM J.\ Sci.\ Comput., 27 (2005), pp.~995--1013.

\bibitem{QuMaNe16}
{\sc A.~Quarteroni, A.~Manzoni, and F.~Negri}, {\em Reduced basis methods for
  partial differential equations}, vol.~92, Springer, Cham, 2016.

\bibitem{Shu2014}
{\sc C.-W. Shu}, {\em Discontinuous {G}alerkin method for time-dependent
  problems: Survey and recent developments}, in Recent Developments in
  Discontinuous {G}alerkin Finite Element Methods for Partial Differential
  Equations: 2012 John H Barrett Memorial Lectures, X.~Feng, O.~Karakashian,
  and Y.~Xing, eds., Springer, Cham, 2014, pp.~25--62.

\bibitem{SmeOhl2017}
{\sc K.~Smetana and M.~Ohlberger}, {\em Hierarchical model reduction of
  nonlinear partial differential equations based on the adaptive empirical
  projection method and reduced basis techniques}, ESAIM: M2AN, 51 (2017),
  pp.~641--677.

\bibitem{MR2891112}
{\sc K.~Urban and A.~Patera}, {\em A new error bound for reduced basis
  approximation of parabolic partial differential equations}, C. R. Math. Acad.
  Sci. Paris, 350 (2012), pp.~203--207.

\bibitem{MR3194123}
\leavevmode\vrule height 2pt depth -1.6pt width 23pt, {\em An improved error
  bound for reduced basis approximation of linear parabolic problems}, Math.
  Comp., 83 (2014), pp.~1599--1615.

\bibitem{Welper2013}
{\sc G.~Welper}, {\em Infinite dimensional stabilization of
  convection-dominated problems}, PhD thesis, Hochschulbibliothek der
  Rheinisch-Westf{\"a}lischen Technischen Hochschule Aachen, 2013.

\bibitem{XZ94}
{\sc J.~Xu and L.~Zikatanov}, {\em Some observations on {B}abu\v ska and
  {B}rezzi theories}, Numer. Math., 94 (2003), pp.~195--202.

\bibitem{ZaNo16}
{\sc O.~Zahm and A.~Nouy}, {\em Interpolation of inverse operators for
  preconditioning parameter-dependent equations}, SIAM J.\ Sci.\ Comput., 38
  (2016), pp.~A1044--A1074.

\bibitem{Zetal11}
{\sc J.~Zitelli, I.~Muga, L.~Demkowicz, J.~Gopalakrishnan, D.~Pardo, and V.~M.
  Calo}, {\em A class of discontinuous {P}etrov-{G}alerkin methods. {P}art
  {IV}: the optimal test norm and time-harmonic wave propagation in 1{D}}, J.
  Comput. Phys., 230 (2011), pp.~2406--2432.
\end{thebibliography}
\end{document}